\documentclass[reqno,11pt]{amsart}
\usepackage{amsfonts, amsmath, amssymb, amsthm, color}
\allowdisplaybreaks[1]
\newtheorem{thm}{Theorem}[section]
\newtheorem{lemma}{Lemma}[section]
\newtheorem{prop}{Proposition}[section]
\newtheorem{cor}{Corollary}[section]

\newtheorem{rmk}{Remark}[section]
\theoremstyle{definition}
\numberwithin{equation}{section}
\newcommand{\rr}{\mathbb{R}}
\newcommand{\al}{\alpha}
\newcommand{\de}{\delta}
\newcommand{\ga}{\gamma}

\newcommand{\la}{\lambda}

\newcommand{\ur}{u_\rho}
\newcommand{\Wr}{W_\rho}
\newcommand{\phr}{\varphi_\rho}
\newcommand{\wda}{w_\delta^\alpha}
\newcommand{\wdai}{w_{\delta_i}^{\alpha_i}}

\newcommand{\wi}{w_i}
\newcommand{\wj}{w_j}
\newcommand{\Aj}{A_j}
\newcommand{\ai}{{\alpha_i}}
\newcommand{\dei}{\delta_i}
\newcommand{\aj}{{\alpha_j}}
\newcommand{\dej}{\delta_j}
\newcommand{\Thj}{\Theta_j}

\newcommand{\si}{\sigma(i)}

\newcommand{\kj}{k_j}
\newcommand{\Akcal}{\mathcal A_k}

\newcommand{\qj}{q_j}
\newcommand{\bgk}{\beta_{\gamma,k}}
\newcommand{\Rrho}{\mathcal R_\rho}
\newcommand{\Srho}{\mathcal S_\rho}
\newcommand{\Lrho}{\mathcal L_\rho}
\newcommand{\Nrho}{\mathcal N_\rho}
\newcommand{\Trho}{\mathcal T_\rho}
\newcommand{\phir}{\phi_\rho}
\newcommand{\phij}{\phi^j}
\newcommand{\Omj}{\Omega_n^j}

\newcommand{\scp}{\sigma}
\newcommand{\scpi}{\sigma_i}
\newcommand{\scpj}{\sigma_j}
\newcommand{\Zi}{Z_i}
\newcommand{\Rj}{R_j}
\newcommand{\calH}{\mathcal H_\gamma}
\newcommand{\calMk}{\mathcal M_k}
\newcommand{\msp}{m_+}
\newcommand{\msm}{m_-}
\newcommand{\tmsp}{\widetilde m_+}
\newcommand{\tmsm}{\widetilde m_-}
\newcommand{\Ep}{\mathcal E_+}
\newcommand{\Em}{\mathcal E_-}
\newcommand{\tp}{\tau_1}

\newcommand{\ka}{\kappa}
\begin{document}
\title[Asymmetric sinh-Poisson equation]{Sign-changing tower of bubbles for a sinh-Poisson equation with asymmetric exponents}
\author[A.~Pistoia]{A.~Pistoia}
\address[A.~Pistoia] {Dipartimento SBAI, Universit\`{a} di Roma \lq\lq La Sapienza\rq\rq, Via Antonio Scarpa 16, 00161 Rome, Italy}
\email{pistoia@dmmm.uniroma1.it}
\author[T.~Ricciardi]{T.~Ricciardi}
\address[T.~Ricciardi] {Dipartimento di Matematica e Applicazioni ``R.~Caccioppoli", Universit\`{a} di Napoli Federico II, 
Via Cintia, 80126 Naples, Italy}
\email{tonricci@unina.it}
\begin{abstract}
Motivated by the statistical mechanics description of stationary 
2D-turbulence,
for a sinh-Poisson type equation with
asymmetric nonlinearity, 
we construct a concentrating solution sequence in the form
of a tower of singular Liouville bubbles, each of which has a 
different degeneracy exponent. 
The asymmetry parameter $\gamma\in(0,1]$ corresponds to the ratio between the intensity of the negatively rotating vortices
and the intensity of the positively rotating vortices.
Our solutions correspond to a superposition of highly concentrated vortex configurations of
alternating orientation; they extend in a nontrivial way some known results for $\ga=1$.
Thus, by analyzing the case  
$\ga\neq1$ we
emphasize specific properties of the physically relevant parameter $\ga$ in the 
vortex concentration phenomena.
\end{abstract}
\date{}
\maketitle
\section{Introduction and statement of the main result}
\label{sec:intro}
We are interested in the existence of bubble-tower type solutions for the problem:
\begin{equation}
\label{eq:pb}
\left\{
\begin{aligned}
-\Delta u=&\rho(e^u-\tau e^{-\gamma u})&&\mbox{in\ }\Omega\\
u=&0&&\mbox{on\ }\partial\Omega,
\end{aligned}
\right.
\end{equation}
where $\Omega\subset\mathbb R^2$ is a smooth bounded domain,
$\rho>0$ is a small constant, $\gamma,\tau\in(0,1]$.
\par
Equation~\eqref{eq:pb} arises in the statistical mechanics description of two-dimensional equilibrium turbulence,
as initiated by Onsager~\cite{Onsager}.
More precisely, in an unpublished manuscript reproduced in the review article \cite{EyinkSreenivasan},
Onsager derived the following equation (see also \cite{SawadaSuzuki} for a rigorous derivation):
\begin{equation}
\label{eq:Onsager}
\left\{
\begin{aligned}
-\Delta u=&\la\left(\tp\frac{e^u}{\int_\Omega e^u\,dx}
-(1-\tp)\ga\frac{e^{-\gamma u}}{\int_\Omega e^{-\ga u}\,dx}\right)&&\mbox{in\ }\Omega\\
u=&0&&\mbox{on\ }\partial\Omega,
\end{aligned}
\right.
\end{equation} 
where $u$ denotes the stream function of the two-dimensional flow, $\la>0$ is a constant
related to the inverse temperature, 
the positively rotating vortices have unit intensity,
$\ga\in(0,1]$ denotes the intensity
of the negatively rotating vortices
and $\tp\in[0,1]$ determines a priori the ratio of the number of positively rotating vortices to
the total number of vortices.
In more recent years, a similar equation was derived by Neri~\cite{Neri},
under the assumption that the vortex intensities are independent identically distributed random variables
with probability measure $\mathcal P$, defined on the (normalized) vortex intensity range $[-1,1]$.
If such a measure is chosen in the form $\mathcal P(dr)=\tp\delta_1(dr)+(1-\tp)\delta_{-\ga}(dr)$, 
where $\delta_1(dr),\delta_{-\ga}(dr)\in\mathcal M([-1,1])$ denote Dirac measures, the resulting equation 
reduces to:
\begin{equation}
\label{eq:Neri}
\left\{
\begin{aligned}
-\Delta u=&\la\frac{\tp e^u-(1-\tp)\ga e^{-\ga u}}
{\int_\Omega(\tp e^u+(1-\tp)e^{-\ga u})\,dx}
&&\mbox{in\ }\Omega\\
u=&0&&\mbox{on\ }\partial\Omega.
\end{aligned}
\right.
\end{equation} 
\par
We observe that the limit case $\tau=0$ in \eqref{eq:pb} yields the well-known Gelfand problem
\begin{equation}
\label{eq:Gelfand}
\begin{aligned}
&-\Delta u=\rho e^u\quad\mbox{in\ }\Omega, 
&&u=0\quad\mbox{on\ }\partial\Omega
\end{aligned}
\end{equation}
and correspondingly the limit case $\tp=1$ in \eqref{eq:Onsager} and \eqref{eq:Neri} yields the
so-called standard mean field equation 
\begin{equation}
\label{eq:meanfield}
\begin{aligned}
&-\Delta u=\la\frac{e^u}{\int_\Omega e^u\,dx}\quad\mbox{in\ }\Omega, 
&&u=0\quad\mbox{on\ }\partial\Omega.
\end{aligned}
\end{equation}
There is a vast literature concerning \eqref{eq:Gelfand}--\eqref{eq:meanfield}, 
see, e.g., \cite{CagliotiLionsMarchioroPulvirenti, JosephLundgren, Lin, Malchiodi} and the references therein.
\par
In the special case $\ga=1$, $\tau=1$ problem~\eqref{eq:pb}
reduces to the sinh-Poisson problem
\begin{equation}
\label{eq:sinhPoisson}
\begin{aligned}
&-\Delta u=\rho(e^u-e^{-u})\quad\mbox{in\ }\Omega,
&&u=0\quad\mbox{on\ }\partial\Omega,
\end{aligned}
\end{equation}
while the non-local counterparts \eqref{eq:Onsager} and \eqref{eq:Neri}
of problem~\eqref{eq:pb} with $\ga=1$
are equivalent to the problems
\begin{equation}
\label{eq:sinhOnsager}
\begin{aligned}
&-\Delta u=\la_1\frac{e^u}{\int_\Omega e^u\,dx}
-\la_2\frac{e^{-u}}{\int_\Omega e^{-u}\,dx}
\quad\mbox{in\ }\Omega,
&&u=0\quad\mbox{on\ }\partial\Omega,
\end{aligned}
\end{equation}
and
\begin{equation}
\label{eq:sinhNeri}
\begin{aligned}
&-\Delta u=\frac{\la_1 e^u-\la_2 e^{-u}}{\int_\Omega(\la_1 e^u+\la_2 e^{-u})\,dx}
\quad\mbox{in\ }\Omega,
&&u=0\quad\mbox{on\ }\partial\Omega,
\end{aligned}
\end{equation}
respectively.
Problem~\eqref{eq:sinhOnsager} was derived in \cite{JoyceMontgomery}--\cite{PointinLundgren}
by statistical mechanics arguments.
The sinh-Poisson equation~\eqref{eq:sinhPoisson} is also related to constant mean curvature surfaces, see \cite{JostWangYeZhou}.
Problems~\eqref{eq:sinhPoisson}--\eqref{eq:sinhOnsager}--\eqref{eq:sinhNeri} received a considerable attention in recent years,
see \cite{GrossiPistoia, OhtsukaSuzuki, JostWangYeZhou, Ricciardi2007, Zhou, Jevnikar}. 
\par
On the other hand, few results are available for \eqref{eq:pb}. The existence of concentrating sign-changing solutions
was recently established in \cite{PistoiaRicciardi2016} and mountain pass solutions were obtained in \cite{RicciardiZecca052016}.
The special 
case $\ga=1/2$ was studied in \cite{JevnikarYang2016} in relation to the Tzitz\'eica  equation in affine geometry.

\par
Our aim in this article is to construct a family of solutions $u_\rho$ to problem~\eqref{eq:pb} which concentrate 
as $\rho\to0^+$ with an arbitrarily prescribed number 
$k\in\mathbb N$ of sign-changing
singular bubbles, on the line of \cite{GrossiPistoia}.
\par
We recall that $m,n\in\mathbb N$ are \textit{coprime} if they do not admit common divisors.
We make the following assumptions for the domain~$\Omega$:
\begin{equation}
\label{assumpt:Omega}
\begin{aligned}
0\in\Omega\ \mbox{and\ }
&\begin{cases}
x\in\Omega\Rightarrow -x\in\Omega\ \mbox{and\ }x\,e^{2\pi\sqrt{-1}/(m+n)}\in\Omega,
&\mbox{if\ }\ga=\frac{m}{n},\ m,n\in\mathbb N\ \mbox{coprime};\\
x\in\Omega\Rightarrow -x\in\Omega,
&\mbox{if\ }\ga\not\in\mathbb Q,
\end{cases}
\\
\end{aligned}
\end{equation}
where, in complex notation, multiplication by $e^{2\pi\sqrt{-1}/(m+n)}$ denotes a rotation
about the origin by the angle $2\pi/(m+n)$.
\par
Correspondingly, we define the Sobolev space
\begin{equation}
\label{def:H}
\calH=
\begin{cases}
\{\varphi\in H_0^1(\Omega):\ \varphi(xe^{2\pi\sqrt{-1}/(m+n)})=\varphi(x)=\varphi(-x)\,\forall x\in\Omega\},
&\mbox{if\ }
\left\{
\begin{aligned}
&\ga=m/n,\\ &m,n\in\mathbb N\ \mbox{coprime}
\end{aligned}
\right.
\\
\{\varphi\in H_0^1(\Omega):\ \varphi(-x)=\varphi(x)\,\forall x\in\Omega\},
&\mbox{if\ }\ga\not\in\mathbb Q.
\end{cases}
\end{equation}
We establish the following result.
\begin{thm}
\label{thm:main}
Fix $\ga\in(0,1]$. Assume that $\Omega$ satisifes the symmetry assumption~\eqref{assumpt:Omega}.
For any $k\in\mathbb N$ there exists $\rho_0>0$ such that for all $\rho\in(0,\rho_0)$
problem~\eqref{eq:pb}
admits a concentrating sign-changing family of solutions $\ur\in\calH$ satisfying
\begin{equation*}
\begin{aligned}
&\rho e^{\ur}\,dx\stackrel{\ast}{\rightharpoonup}m_+(0)\,\delta_0(dx),
&&\rho\tau e^{-\ga\ur}\,dx\stackrel{\ast}{\rightharpoonup}m_-(0)\,\delta_0(dx)
\qquad\mbox{as\ }\rho\to0^+
\end{aligned}
\end{equation*}
weakly in the sense of measures, where the \lq\lq blow-up masses"
$m_+(0),m_-(0)$ are given by
\begin{equation}
\label{def:blowupmassodd}
\left\{
\begin{aligned}
&\frac{m_+(0)}{4\pi}=
\frac{k+1}{2}\left[(1+\frac{1}{\ga})k+1-\frac{1}{\ga}\right]\\
&\frac{m_-(0)}{4\pi}=
\frac{k-1}{2}\left[(1+\frac{1}{\ga})k+1-\frac{1}{\ga}\right]
\end{aligned}
\right.,
\qquad\qquad\mbox{if $k$ is odd;}
\end{equation}
and
\begin{equation}
\label{def:blowupmasseven}
\left\{
\begin{aligned}
&\frac{m_+(0)}{4\pi}=
k\left[(1+\frac{1}{\ga})\frac{k}{2}-\frac{1}{\ga}\right]\\
&\frac{m_-(0)}{4\pi}=
k\left[(1+\frac{1}{\ga})\frac{k}{2}+1\right]
\end{aligned}
\right.,
\qquad\qquad\mbox{if $k$ is even.}
\end{equation}
Moreover,
\begin{equation}
\label{eq:urprofile}
\ur(x)\to\calMk G(x,0) 
\end{equation}
uniformly on compact subsets of $\Omega\setminus\{0\}$ as $\rho\to0^+$,
where $\calMk=m_+(0)-m_-(0)$ is the \lq\lq algebraic total mass" with values
\begin{equation}
\label{def:Mk}
\calMk=\begin{cases}
4\pi[(1+\frac{1}{\ga})k-\frac{1}{\ga}+1],&\mbox{if $k$ is odd;}\\
-4\pi(1+\frac{1}{\ga})k,&\mbox{if $k$ is even.}
\end{cases}
\end{equation}
\end{thm}
We shall obtain the solution $\ur$ in the form $\ur=\Wr+\phir$,
where $\Wr$ is an alternating sum of $k$ singular Liouville bubbles of the form
\[
w_\de^\al(x)=\ln\frac{2\al^2\de^\al}{(\de^\al+|x|^\al)^2},
\]
projected onto $H_0^1(\Omega)$, namely:
\begin{equation*}
\label{eq:ansatz}
\Wr=
\sum_{i=1}^k(-1)^{i-1}\frac{P\wdai}{\ga^{\sigma(i)}},
\end{equation*}
where $\si=(1-(-1)^{i-1})/2$, $i=1,2,\ldots,k$  and where the singularity parameters $\ai\ge 2$ and the concentration parameters $\dei>0$
are suitably chosen in order to ensure that $\|\nabla\phir\|_{L^2(\Omega)}=o(1)$ as $\rho\to0^+$,
see Section~\ref{sec:Ansatz} for the precise statements.
The functions $w_\de^\al$ are solutions to the singular Liouville equation
\begin{equation*}
-\Delta w_\de^\al=|x|^{\al-2}e^{w_\de^\al},
\qquad
\int_{\mathbb R^2}|x|^{\al-2}e^{w_\de^\al}\,dx<+\infty.
\end{equation*}
Since the appropriate choice of $\ai$'s
leads to $\ai\neq\aj$ for $i\neq j$, we find that the bubble tower approximate solution
$\Wr$ is actually the sum of solutions to \textit{different} singular Liouville problems.
Such new blow-up profiles were observed in \cite{GrossiGrumiauPacella2014}.
Towers of concentrated solutions to different singular Liouville equations were initially introduced
in the article \cite{GrossiPistoia},
where the case $\ga=1$ is considered, and which is the main motivation to this work.
\par
The blow-up masses $m_+(0),m_-(0)$ satisfy the identity  
\begin{equation}
\label{eq:massid}
8\pi\left[m_+(0)+\frac{m_-(0)}{\ga}\right]=[m_+(0)-m_-(0)]^2=\calMk^2.
\end{equation}
Moreover, in view of \eqref{assumpt:Omega}, $0\in\Omega$ is a critical point for the
Robin's function.
In fact, identity~\eqref{eq:massid} is a general property for concentrating solution
sequences for \eqref{eq:pb}, and if the concentration occurs at a single point,
such a point is necessarily a critical point for Robin's function, see, e.g.,
Remark~\ref{rmk:massid} for a proof.
For $\ga=1$ identity~\eqref{eq:massid} was derived in \cite{OhtsukaSuzuki}.
\par
It is natural to conjecture that the blow-up mass values \eqref{def:blowupmassodd}--\eqref{def:blowupmasseven}
are the only admissible values for $m_+(0),m_-(0)$, in view of the mass quantization 
results for the case $\ga=1$ in \cite{JostWangYeZhou}.
In this respect, a mass quantization property for \eqref{eq:pb} was announced in \cite{Takahashiseminar};
a partial result in this direction concerning the minimum values for blow-up masses was obtained in \cite{RicciardiZecca052016}. 
\par
From the physics interpretation point of view, the solutions $\ur$ as obtained in Theorem~\ref{thm:main}
yield solutions to \eqref{eq:Onsager} with total mass
\begin{equation}
\label{eq:lambda}
\la=m_+(0)+\frac{m_-(0)}{\ga}=\frac{\calMk^2}{8\pi}=
\begin{cases}
2\pi[(1+\frac{1}{\ga})k+1-\frac{1}{\ga}]^2,&\mbox{if $k$ is odd;}\\
2\pi(1+\frac{1}{\ga})^2k^2,&\mbox{if $k$ is even}
\end{cases}
\end{equation}
and vortex distribution parameter
\begin{equation}
\label{eq:Ontau}
\tp=\frac{m_+(0)}{\la}=\begin{cases}
\frac{k+1}{(1+\frac{1}{\ga})k+1-\frac{1}{\ga}},
&\mbox{if $k$ is odd;}\\
\frac{(1+\frac{1}{\ga})k-\frac{2}{\ga}}{(1+\frac{1}{\ga})^2k},
&\mbox{if $k$ is even.}
\end{cases}
\end{equation}
They also yield solutions to \eqref{eq:Neri} with total mass given 
by \eqref{eq:lambda} and with no restriction on $\tp$.
It may be interesting to note that the \lq\lq total mass" $\lambda$
is the quantity on the left hand side in the identity~\eqref{eq:massid}.
A proof of these statements is provided in the Appendix.
\par
As already mentioned, our approach to prove Theorem~\ref{thm:main} is strongly 
inspired by the singular bubble-tower construction in \cite{GrossiPistoia},
where the case $\ga=1$ is considered, in the $L^p$-framework introduced in \cite{EspositoGrossiPistoia},
see also \cite{DelPinoKowalczykMusso}.
Nevertheless, the case $\ga\neq1$ turns out to be significantly more delicate to handle,
and it emphasizes specific analytic and geometric properties of the asymmetry parameter.
In fact, the dependence of the singularity coefficients $\ai$ and of the concentration parameters $\dei^{\ai}$
on $\ga$ 
is rather subtle; in particular, unlike the case $\ga=1$,
the $\ai$'s are never monotonically increasing with respect to $i$ and
the concentration parameters $\dei^{\ai}$ do not depend linearly with respect to $i$.
Consequently, new ingredients are required in several estimates.
Finally, it is interesting to observe that the geometrical symmetry condition~\eqref{assumpt:Omega}
required for $\Omega$, which ensures invertibility of the linearized operator,
depends in a relevant way on $\ga$, if $\ga\in(0,1]\cap\mathbb Q$. 
\subsection*{Notation}
For any measurable set $A\subset\Omega$ we denote by $\chi_A$ the characteristic function of $A$.
We denote by $C>0$ a general constant whose value may vary from line to line.
When the integration variable is clear from the context, we omit it.
For all $\phi\in H_0^1(\Omega)$ we set $\|\phi\|:=\|\nabla\phi\|_{L^2(\Omega)}$.
\section{Ansatz and idea of the proof}
\label{sec:Ansatz}
We recall that the \lq\lq singular Liouville bubbles" are defined for $\al\ge 2$ and $\de>0$
by
\begin{equation*}
w_\de^\al(x)=\ln\frac{2\al^2\de^\al}{(\de^\al+|x|^\al)^2}.
\end{equation*}
The functions $w_\de^\al$ satisfy
\begin{equation}
\label{eq:w}
-\Delta w_\de^\al=|x|^{\al-2}e^{w_\de^\al},
\qquad
\int_{\mathbb R^2}|x|^{\al-2}e^{w_\de^\al}\,dx<+\infty.
\end{equation}
Furthermore,
\begin{equation}
\label{eq:Liouvillemass}
\int_{\mathbb R^2}|x|^{\al-2}e^{w_\de^\al}\,dx=4\pi\al.
\end{equation}
The functions $\wda$ are uniquely determined as radial solutions for \eqref{eq:w},
see \cite{PrajapatTarantello}.
We denote by $Pw_\de^\al$ the projection of $w_\de^\al$ onto $H_0^1(\Omega)$.
\\
We define
\begin{equation*}
\sigma(i)=
\frac{1-(-1)^{i-1}}{2}
=\begin{cases}
0, 
&\mbox{if\ }i\ \mbox{is odd,}\\
1, 
&\mbox{if\ }i\ \mbox{is even,}\\
\end{cases}
\end{equation*}
for $i=1,2,\ldots,k$.
\subsection*{Ansatz.}
The solutions are of the form $\ur=\Wr+\phr$, where
\begin{equation}
\label{def:Wr}
\Wr=
\sum_{i=1}^k(-1)^{i-1}\frac{P\wdai}{\ga^{\sigma(i)}}
=\sum_{\stackrel{1\le i\le k}{i\ \mathrm{odd}}}P\wdai
-\frac{1}{\ga}\sum_{\stackrel{1\le i\le k}{i\ \mathrm{even}}}P\wdai,
\end{equation}
with 
\begin{equation*}
\ai=
\begin{cases}
2[(1+\frac{1}{\ga})i-\frac{1}{\ga}],&\mbox{if $i$ is odd}\\
2[(1+\ga)i-1],&\mbox{if $i$ is even}
\end{cases}
\end{equation*}
and
\begin{equation*}
\dei=d_i\rho^{s_i}
\qquad\mbox{for some }d_i>0,\quad i=1,2,\ldots,k,
\end{equation*}
where
\begin{equation*}
s_i=
\begin{cases}
\frac{(1+\ga)(k-i)+\ga}{2[(1+\ga)i-1]},&\mbox{if $k$ is odd};\\
\frac{(1+\ga)(k-i)+1}{2[(1+\ga)i-1]},&\mbox{if $k$ is even}.
\end{cases}
\end{equation*}
Note that in particular
\begin{equation*}
\begin{aligned}
&\ai\ge 2, &&\mbox{for all }i=1,2,\ldots,k\\
&s_i>s_{i+1}\ \mbox{and therefore\ }\de_{i}=o(\de_{i+1}),&&\mbox{for all } i=1,2,\ldots,k-1.
\end{aligned}
\end{equation*}
See Section~\ref{sec:aidei} for the precise values of $d_i$, $i=1,2,\ldots,k$
and for the precise power decay rate of $\dei/\de_{i+1}$.
\par
Henceforth, we denote $w_i:=w_{\de_i}^{\al_i}$, $i=1,2,\ldots k$.
\par
The most delicate part of the construction will be to show that 
if $\ai,\dei$ are chosen according to the above definitions, then $\Wr$ approximates a genuine solution to \eqref{eq:pb}
up to an error which vanishes as a \textit{power} of $\rho$, as $\rho\to0^+$ .
This fact, combined with the $|\ln\rho|$-estimate 
for the norm of the linearized operator (see \eqref{eq:linearest} below for the precise statement)
will enable us to obtain the desired solution as the fixed point of a contraction mapping.
\par 
More precisely, let 
\begin{equation*}
f(t):=e^t-\tau e^{-\ga t},\quad t\in\mathbb R.
\end{equation*}
Then, the error term to be estimated is given by:
\begin{equation}
\label{def:R}
\begin{aligned}
\Rrho=&\Delta\Wr+\rho f(\Wr)\\
=&\rho e^{\Wr}-\sum_{\stackrel{1\le i\le k}{i\ \mathrm{odd}}}|x|^{\ai-2}e^{\wi}
-\Big(\rho\tau e^{-\ga\Wr}-\frac{1}{\ga}\sum_{\stackrel{1\le i\le k}{i\ \mathrm{even}}}|x|^{\ai-2}e^{\wi}\Big)
\end{aligned}
\end{equation}
It is convenient to set:
\begin{equation*}
\begin{aligned}
\Ep:=&\rho e^{\Wr}-\sum_{\stackrel{1\le i\le k}{i\ \mathrm{odd}}}|x|^{\ai-2}e^{\wi}\\
\Em:=&\Big(\rho\tau e^{-\ga\Wr}-\frac{1}{\ga}\sum_{\stackrel{1\le i\le k}{i\ \mathrm{even}}}|x|^{\ai-2}e^{\wi}\Big)
\end{aligned}
\end{equation*}
so that
\[
\Rrho=\Ep-\Em.
\]
One of the main technical issues will be to show that,
provided $\ai, \dei$ are chosen as above, 
there exist $p>1$, $\overline\beta=\overline\beta(k,\ga,p)>0$ such that
\begin{equation}
\label{est:EpEm}
\|\Ep\|_{L^p(\Omega)}+\|\Em\|_{L^p(\Omega)}=O(\rho^{\overline\beta}).
\end{equation}
The appropriate choice of the parameters $\ai,\dei$ is carried out
in Section~\ref{sec:aidei}, where some properties necessary for the subsequent estimates
are also derived. 
Then, estimate~\eqref{est:EpEm} is established in Section~\ref{sec:Thetaj}
and Section~\ref{sec:RrhoSrho}.
\par
In order to prove estimate~\eqref{est:EpEm}, we 
define the \textit{shrinking annuli}
\begin{equation*}
\Aj=\{x\in\Omega:\ \sqrt{\de_{j-1}\de_j}\le|x|<\sqrt{\de_j\de_{j+1}}\},
\qquad j=1,2,\ldots,k,
\end{equation*}
where we set $\de_0=0$ and $\de_{k+1}=+\infty$.
\par
We decompose $\Ep$:
\begin{equation*}
\begin{aligned}
\Ep=&\sum_{1\le j\le k}\Ep\chi_{\Aj}\\
=&\sum_{\stackrel{1\le j\le k}{j\ \mathrm{odd}}}\Big(\rho e^{\Wr}-\sum_{\stackrel{1\le i\le k}{i\ \mathrm{odd}}}|x|^{\ai-2}e^{\wi}\Big)\chi_{\Aj}
+\sum_{\stackrel{1\le j\le k}{j\ \mathrm{even}}}\Big(\rho e^{\Wr}-\sum_{\stackrel{1\le i\le k}{i\ \mathrm{odd}}}|x|^{\ai-2}e^{\wi}\Big)\chi_{\Aj}\\
=&\sum_{\stackrel{1\le j\le k}{j\ \mathrm{odd}}}\Big(\rho e^{\Wr}-|x|^{\aj-2}e^{\wj}\Big)\chi_{\Aj}
+\sum_{\stackrel{1\le j\le k}{j\ \mathrm{even}}}\rho e^{\Wr}\chi_{\Aj}
-\sum_{\stackrel{1\le j\le k}{j\neq i}}\sum_{\stackrel{1\le i\le k}{i\ \mathrm{odd}}}|x|^{\ai-2}e^{\wi}\chi_{\Aj}\\
=:&\Ep^1+\Ep^2+\Ep^3.
\end{aligned}
\end{equation*}
Similarly, we decompose $\Em$:
\begin{equation*}
\begin{aligned}
\Em=&\sum_{1\le j\le k}\Em\chi_{\Aj}\\
=&\sum_{\stackrel{1\le j\le k}{j\ \mathrm{even}}}\Big(\rho\tau e^{-\ga\Wr}-\frac{1}{\ga}\sum_{\stackrel{1\le i\le k}{i\ \mathrm{even}}}|x|^{\ai-2}e^{\wi}\Big)\chi_{\Aj}
+\sum_{\stackrel{1\le j\le k}{j\ \mathrm{odd}}}\Big(\rho\tau e^{-\ga\Wr}-\frac{1}{\ga}\sum_{\stackrel{1\le i\le k}{i\ \mathrm{even}}}|x|^{\ai-2}e^{\wi}\Big)\chi_{\Aj}\\
=&\sum_{\stackrel{1\le j\le k}{j\ \mathrm{even}}}\Big(\rho\tau e^{-\ga\Wr}-|x|^{\aj-2}e^{\wj}\Big)\chi_{\Aj}
+\sum_{\stackrel{1\le j\le k}{j\ \mathrm{odd}}}\rho\tau e^{-\ga\Wr}\chi_{\Aj}
-\frac{1}{\ga}\sum_{\stackrel{1\le j\le k}{j\neq i}}\sum_{\stackrel{1\le i\le k}{i\ \mathrm{even}}}|x|^{\ai-2}e^{\wi}\chi_{\Aj}\\
=:&\Em^1+\Em^2+\Em^3.
\end{aligned}
\end{equation*}
In short, the choice of $\ai,\dei$ will ensure the smallness of the error terms $\Ep^1,\Em^1$,
which measure the interaction in the $j$-th annulus $\Aj$ between the $j$-th bubble $w_{\dej}^{\aj}$ and all other bubbles.
Indeed, as in \cite{GrossiPistoia}, the errors $\Ep^1,\Em^1$ are small inside the $j$-th annulus $\Aj$ \lq\lq because" 
the choice of $\aj$ will cancel the interaction of
the $j$-th bubble and all previous (faster concentrating) bubbles,
whereas the choice of $\dej$ will cancel the interaction of the $j$-th bubble $w_{\dej}^{\aj}$ and all subsequent (slower concentrating)
bubbles. 
On the other hand, the error terms $\Ep^2,\Em^2$ are estimated by some delicate recursive relations
for $\ai,\dei$.
Estimation of $\Ep^3,\Em^3$ follows from the fact that,
outside the $j$-th annulus $\Aj$, the $j$-th bubble $w_{\dej}^{\aj}$ is negligible,
up to an error which vanishes as a power of $\rho$.
\par
Once \eqref{est:EpEm} is established, we define
\begin{equation}
\label{def:RSN}
\begin{aligned}
\Srho:=&\rho f'(\Wr)-\sum_{i=1}^k|x|^{\ai-2}e^{\wi}\\
\Nrho(\phi):=&\rho[f(\Wr+\phi)-f(\Wr)-f'(\Wr)\phi]\\
\Lrho\phi:=&-\Delta\phi-\sum_{i=1}^k|x|^{\ai-2}e^{\wi}.
\end{aligned}
\end{equation}
We note that
\[
\Srho=\Ep+\ga\Em,
\]
so that estimate~\eqref{est:EpEm} provides an estimate for $\Srho$ as well.
In Section~\ref{sec:lineartheory}
we show that for any $p>1$ there exists $c>0$ such that
\begin{equation}
\label{eq:linearest}
\|\phi\|\le c|\ln\rho|\|\Lrho\phi\|_p,
\qquad\forall \phi\in\calH.
\end{equation}
At this point, we can show that there exists $\rho_0>0$ such that
for all $\rho\in(0,\rho_0)$ the equation
\begin{equation*}
\phi=\mathcal T(\phi):=(\Lrho)^{-1}(\Nrho(\phi)+\Srho\phi+\Rrho)
\end{equation*}
admits a fixed point $\phir$ satisfying $\|\phir\|\le R\rho^{\overline\beta_p}|\ln\rho|$
for some $\overline\beta_p=\overline\beta_p(\tau,\ga,k)>0$, $p>1$ and $R>0$. 
The function $\ur=\Wr+\phir$ is the desired
solution to \eqref{eq:pb}.
The details of the fixed point argument are contained in Section~\ref{sec:fixpoint}.
\section{Definition and properties of the parameters}
\label{sec:aidei}
In this section we define the parameters $\aj,\dej$, $j=1,2,\ldots,k$ and we establish some properties which will be used
in order to estimate the error terms.
The justification of the choice of $\aj,\dej$ will be provided in Section~\ref{sec:Thetaj}.
\par
We denote by $G(x,y)$, $x,y\in\Omega$, $x\neq y$, the Green's function for $\Omega$,
namely
\[
-\Delta G(\cdot, y)=\delta_y\quad \mbox{in\ $\Omega$},
\qquad\qquad G(\cdot, y)=0\quad \mbox{on\ }\partial\Omega.
\]
We denote by $H(x,y)$ the regular part of $G(x,y)$:
\begin{equation}
\label{def:regG}
G(x,y)=\frac{1}{2\pi}\ln\frac{1}{|x-y|}+H(x,y).
\end{equation}
We set 
\begin{equation*}
h(x)=H(x,0).
\end{equation*}
\subsection*{Definition of $\aj,\dej$.}
The appropriate values for $\aj,\dej$ are deduced form the following defining conditions.
\subsubsection*{Odd index}
If $j\in\{1,2,\ldots,k\}$ is odd, we define:
\begin{equation}
\label{def:adodd}
\left\{
\begin{split}
&\aj:=2\left(1+\sum_{i<j}\frac{(-1)^i}{\ga^{\sigma(i)}}\al_i\right)\\
&\ln(2\aj^2\dej^{\aj}):=2\sum_{i>j}\frac{(-1)^i}{\ga^{\sigma(i)}}\ln\dei^{\ai}-4\pi h(0)\sum_{i=1}^k\frac{(-1)^{i}}{\ga^{\sigma(i)}}\ai
+\ln\rho.
\end{split}
\right.
\end{equation}
\subsubsection*{Even index}
If $j\in\{1,2,\ldots,k\}$ is even, we define:
\begin{equation}
\label{def:adeven}
\left\{
\begin{split}
&\aj:=2\left(1-\ga\sum_{i<j}\frac{(-1)^{i}}{\ga^{\sigma(i)}}\ai\right)\\
&\ln(2\aj^2\dej^{\aj}):=-2\ga\sum_{i>j}\frac{(-1)^{i}}{\ga^{\sigma(i)}}\ln\dei^{\ai}+4\pi\ga h(0)\sum_{i=1}^k\frac{(-1)^{i}}{\ga^{\sigma(i)}}\ai
+\ln(\rho\tau\ga).
\end{split}
\right.
\end{equation}
We note that the $\aj$'s are determined by the number $k$ of bubbles only;
the concentration parameters $\dej$ also depend on $\rho$.
Moreover, \eqref{def:adodd}--\eqref{def:adeven} define $\aj$ recursively in terms of $\al_1,\al_2,\ldots,\al_{j-1}$
and $\dej$ in terms of $\al_1,\al_2,\ldots,\al_k$ and $\de_{j+1},\de_{j+2}\ldots,\de_{k}$.
\begin{rmk}
There holds $\aj\ge2$ for all $j=1,2,\ldots,k$.
\end{rmk}
An explicit computation yields the following first values of the $\aj$'s:
\begin{equation}
\label{eq:aifirstvalues}
\begin{aligned}
&\qquad\qquad\al_1=2,\ \al_3=2(3+\frac{2}{\ga}),\ \al_5=2(5+\frac{4}{\ga}),\ \al_7=2(7+\frac{6}{\ga}),\ldots\\
&\qquad\qquad\al_2=2(1+2\ga),\ \al_4=2(3+4\ga),\ \al_6=2(5+6\ga),\ \al_8=2(7+8\ga)\ldots
\end{aligned}
\end{equation}
Moreover, if $k=2$, we obtain the following decay rates for the $\dej$'s:
\[
\de_1=d_1\rho^{(2+\ga)/(2\ga)},
\qquad\qquad
\de_2=d_2\rho^{1/[2(1+2\ga)]}.
\]
It will be convenient to set
\begin{equation}
\label{def:Ak}
\Akcal:=\sum_{i=1}^k\frac{(-1)^i}{\ga^{\sigma(i)}}\ai.
\end{equation}
\subsection*{Properties of the singularity coefficients $\aj$}
In this subsection we determine $\aj$ explicitly in terms of $j$,
for $j=1,2,\ldots,k$,
and we establish the main properties
of the $\aj$'s which will be needed in the sequel.
\begin{prop}
\label{prop:aimain}
For all $j=1,2,\ldots,k$ we have:
\begin{equation}
\label{def:a}
\aj=\begin{cases}
2[(1+\frac{1}{\ga})j-\frac{1}{\ga}]
&\mbox{if\ $j$ is odd};\\
2[(1+\ga)j-1]
&\mbox{if\ $j$ is even}.
\end{cases}
\end{equation}
Moreover, we have
\begin{equation}
\label{eq:ajsumodd}
\sum_{\stackrel{1\le j\le k}{j\ \mathrm{odd}}}\aj=
\begin{cases}
\frac{k+1}{2}[(1+\frac{1}{\ga})k+1-\frac{1}{\ga}],
&\mbox{if $k$ is odd;}\\
k[(1+\frac{1}{\ga})\frac{k}{2}-\frac{1}{\ga}],
&\mbox{if $k$ is even;}
\end{cases}
\end{equation}
\begin{equation}
\label{eq:ajsumeven}
\frac{1}{\ga}\sum_{\stackrel{1\le j\le k}{j\ \mathrm{even}}}\aj=
\begin{cases}
\frac{k-1}{2}[(1+\frac{1}{\ga})k+1-\frac{1}{\ga}],
&\mbox{if $k$ is odd;}\\
\frac{k}{2}[(1+\frac{1}{\ga})(k+1)+1-\frac{1}{\ga}],
&\mbox{if $k$ is even;}
\end{cases}
\end{equation}
In particular, the following identity holds true.
\begin{equation}
\label{eq:suma}
\Akcal=
\begin{cases}
-(1+\frac{1}{\ga})k+\frac{1}{\ga}-1,&\mbox{if $k$ is odd;}\\
(1+\frac{1}{\ga})k,&\mbox{if $k$ is even,}
\end{cases}
\end{equation}
where $\Akcal$ is defined in \eqref{def:Ak}.
\end{prop}
The following consequence of Proposition~\ref{prop:aimain} will be essential in the proof
of the invertibility of the linearized operator $\Lrho$ defined in \eqref{def:RSN}.
Indeed, the kernel of $\Lrho$, is determined by the
divisibility properties of $\aj/2$, $j=1,2,\ldots,k$.
We recall that two integers $m,n\in\mathbb N$ are said to be \emph{coprime}
if they do not admit common divisors.
\begin{cor}
\label{cor:aiform}
Suppose $\ga=m/n$ with $m,n\in\mathbb N$, $m,n$ coprime, 
and suppose that
$\aj/2\in\mathbb N$ for some $j=1,2,\ldots,k$. Then, there exists $\kj\in\mathbb N\cup\{0\}$ such that
\begin{equation}
\label{eq:aiform}
\frac{\aj}{2}=
\begin{cases}
(m+n)\kj+1,&\mbox{if\ }j\ \mbox{is odd;}\\
(m+n)\kj-1,&\mbox{if\ }j\ \mbox{is even.}
\end{cases}
\end{equation}
\end{cor}
\begin{proof}
Suppose $j$ is odd.
Then, 
\[
\frac{\aj}{2}=(1+\frac{n}{m})j-\frac{n}{m}=j+\frac{n}{m}(j-1).
\]
Since $m,n$ are coprime, it follows that $j-1=\kj m$ for some $\kj\in\mathbb N\cup\{0\}$.
Consequently,
\[
\frac{\aj}{2}=\kj m+1+n\kj=(m+n)\kj+1,
\]
and \eqref{eq:aiform} is established for odd $j$.
\par
Similarly, suppose $j$ is even.
Then,
\[
\frac{\aj}{2}=(1+\frac{m}{n})j-1=\frac{m}{n}j+j-1.
\]
Since $m,n$ are coprime, it follows that $j=\kj n$ for some $\kj\in\mathbb N$.
Consequently,
\[
\frac{\aj}{2}=(1+\frac{m}{n})\kj n-1=(m+n)\kj-1.
\]
Formula \eqref{eq:aiform} is completely established.
\end{proof}
In order to prove Proposition~\ref{prop:aimain} we establish some lemmas.
\begin{lemma}
The following recursive formulae hold:
\begin{equation}
\label{eq:arec}
\aj=\begin{cases}
\frac{1}{\ga}\al_{j-1}+2(1+\frac{1}{\gamma}),&\mbox{if $j$ is odd}\\
\gamma\al_{j-1}+2(1+\gamma),&\mbox{if $j$ is even},
\end{cases}
\end{equation}
for $j=2,3,\ldots,k$.
\end{lemma}
\begin{proof}
Suppose $j$ is odd. Then, in view of \eqref{def:adodd} and the fact that $j-1$ is even, we have
\begin{align*}
\aj=&2\left(1+\sum_{i<j}\frac{(-1)^i}{\ga^{\si}}\ai\right)
=2\left(1+\frac{(-1)^{j-1}}{\ga^{\sigma(j-1)}}\al_{j-1}+\sum_{i<j-1}\frac{(-1)^i}{\ga^{\si}}\ai\right)\\
=&2\left(1+\frac{1}{\ga}\al_{j-1}+\sum_{i<j-1}\frac{(-1)^i}{\ga^{\si}}\ai\right).
\end{align*}
In view of \eqref{def:adeven}, we have
\[
\al_{j-1}=2\left(1-\gamma\sum_{i<j-1}\frac{(-1)^i}{\ga^{\si}}\ai\right)
=2-2\ga\sum_{i<j-1}\frac{(-1)^i}{\ga^{\si}}\ai
\]
and therefore
\[
\sum_{i<j-1}\frac{(-1)^i}{\ga^{\si}}\ai=\frac{1}{\ga}-\frac{1}{2\ga}\al_{j-1}.
\]
We conclude that
\[
\aj=2\left(1+\frac{1}{\ga}\al_{j-1}+\frac{1}{\ga}-\frac{1}{2\ga}\al_{j-1}\right)
=2\left(1+\frac{1}{\ga}+\frac{1}{2\ga}\al_{j-1}\right),
\]
and \eqref{eq:arec} is established for odd indices $j$.
\par
Similarly, suppose $j$ is even. Then, in view of \eqref{def:adeven} and the fact that $j-1$ is odd, 
we have
\begin{align*}
\aj=&2\left(1-\gamma\sum_{i<j}\frac{(-1)^i}{\ga^{\si}}\ai\right)
=2\left(1-\gamma\frac{(-1)^{j-1}}{\ga^{\sigma(j-1)}}\al_{j-1}-\gamma\sum_{i<j-1}\frac{(-1)^i}{\ga^{\si}}\ai\right)\\
=&2\left(1+\ga\al_{j-1}-\gamma\sum_{i<j-1}\frac{(-1)^i}{\ga^{\si}}\ai\right).
\end{align*}
Since $j-1$ is odd, we have from \eqref{def:adodd} that
\[
\al_{j-1}=2\left(1+\sum_{i<j-1}\frac{(-1)^i}{\ga^{\si}}\ai\right)
=2+2\sum_{i<j-1}\frac{(-1)^i}{\ga^{\si}}\ai,
\]
that is,
\[
\sum_{i<j-1}\frac{(-1)^i}{\ga^{\si}}\ai
=\frac{1}{2}\al_{j-1}-1.
\]
We deduce that
\[
\aj=2\left(1+\ga\al_{j-1}-\frac{\ga}{2}\al_{j-1}+\ga\right)
=2(1+\ga)\al_{j-1}+\ga\al_{j-1},
\]
and the recursive formula \eqref{eq:arec} is also established for all even indices $j$.
\end{proof}
We also use the following results,
whose proof is elementary.
\begin{lemma}
\label{lem:sums}
Let $k\in\mathbb N$.
Then,
\begin{equation*}
\begin{aligned}
&\sum_{\stackrel{1\le j\le k}{j\ \mathrm{odd}}}1=
\begin{cases}
\frac{k+1}{2},&\mbox{if $k$ is odd}\\
\frac{k}{2},&\mbox{if $k$ is even}
\end{cases};
&&
\sum_{\stackrel{1\le j\le k}{j\ \mathrm{even}}}1=
\begin{cases}
\frac{k-1}{2},&\mbox{if $k$ is odd}\\
\frac{k}{2},&\mbox{if $k$ is even}
\end{cases}
\end{aligned}
\end{equation*}
and
\begin{equation*}
\begin{aligned}
&
\sum_{\stackrel{1\le j\le k}{j\ \mathrm{odd}}}j=
\begin{cases}
\frac{(k+1)^2}{4},&\mbox{if $k$ is odd}\\
\frac{k^2}{4},&\mbox{if $k$ is even}
\end{cases};
&&
\sum_{\stackrel{1\le j\le k}{j\ \mathrm{even}}}j=
\begin{cases}
\frac{(k-1)(k+1)}{4},&\mbox{if $k$ is odd}\\
\frac{k(k+2)}{4},&\mbox{if $k$ is even}
\end{cases}.
\end{aligned}
\end{equation*}
\end{lemma}
Now we can provide the proof of Proposition~\ref{prop:aimain}.
\begin{proof}[Proof of Proposition~\ref{prop:aimain}]
Proof of \eqref{def:a}.
We argue by induction. We already know from \eqref{eq:aifirstvalues} that
$\al_1=2$ and $\al_2=2(1+2\ga)$.
\par
Suppose \eqref{def:a} holds true for all $i<j$, with $j$ an odd index.
Then, in view of \eqref{eq:arec} and the induction assumption we have:
\[
\aj=\frac{1}{\ga}\al_{j-1}+2(1+\frac{1}{\ga})
=\frac{2}{\ga}[(1+\ga)(j-1)-1]+2(1+\frac{1}{\ga})
=2[(1+\frac{1}{\ga})j-\frac{1}{\ga}],
\]
and \eqref{def:a} is established in this case.
\par
Suppose \eqref{def:a} holds true for all $i<j$, with $j$ an even index.
Then, in view of \eqref{eq:arec} and the induction assumption we have:
\[
\aj=\ga\al_{j-1}+2(1+\ga)
=2\ga[(1+\frac{1}{\ga})(j-1)-\frac{1}{\ga}]+2(1+\ga)
=2[(1+\ga)j-1].
\]
Now \eqref{def:a} is completely established.
\par
Proof of \eqref{eq:ajsumodd}.
Recall from \eqref{def:a} that if $j$ is odd, then $\aj=2[(1+\frac{1}{\ga})j-\frac{1}{\ga}]$.
Hence, we compute:
\begin{equation*}
\sum_{\stackrel{1\le j\le k}{j\ \mathrm{odd}}}\aj
=2(1+\frac{1}{\ga})\sum_{\stackrel{1\le j\le k}{j\ \mathrm{odd}}}j-\frac{2}{\ga}\sum_{\stackrel{1\le j\le k}{j\ \mathrm{odd}}}1.
\end{equation*}
In view of Lemma~\ref{lem:sums}, for $k$ odd we deduce that:
\begin{equation*}
\sum_{\stackrel{1\le j\le k}{j\ \mathrm{odd}}}\aj
=(1+\frac{1}{\ga})\frac{(k+1)^2}{2}-\frac{k+1}{\ga}
=\frac{k+1}{2}[(1+\frac{1}{\ga})k+1-\frac{1}{\ga}].
\end{equation*}
Similarly, for $k$ even we deduce that
\begin{equation*}
\sum_{\stackrel{1\le j\le k}{j\ \mathrm{odd}}}\aj
=(1+\frac{1}{\ga})\frac{k^2}{2}-\frac{k}{\ga}
=k[(1+\frac{1}{\ga})\frac{k}{2}-\frac{1}{\ga}].
\end{equation*}
Proof of \eqref{eq:ajsumeven}.
Recall from \eqref{def:a} that if $j$ is even, then $\aj=2[(1+\ga)j-1]$.
In view of Lemma~\ref{lem:sums}, for $k$ odd we deduce that:
\begin{equation*}
\frac{1}{\ga}\sum_{\stackrel{1\le j\le k}{j\ \mathrm{even}}}\aj
=2(1+\frac{1}{\ga})\sum_{\stackrel{1\le j\le k}{j\ \mathrm{even}}}j
-\frac{2}{\ga}\sum_{\stackrel{1\le j\le k}{j\ \mathrm{even}}}1.
\end{equation*}
In view of Lemma~\ref{lem:sums}, for $k$ odd we deduce that:
\begin{equation*}
\frac{1}{\ga}\sum_{\stackrel{1\le j\le k}{j\ \mathrm{even}}}\aj
=(1+\frac{1}{\ga})\frac{(k-1)(k+1)}{2}-\frac{k-1}{\ga}
=\frac{k-1}{2}[(1+\frac{1}{\ga})k+1-\frac{1}{\ga}].
\end{equation*}
Similarly, for $k$ even we deduce that
\begin{equation*}
\frac{1}{\ga}\sum_{\stackrel{1\le j\le k}{j\ \mathrm{even}}}\aj
=(1+\frac{1}{\ga})\frac{k(k+2)}{2}-\frac{k}{\ga}
=\frac{k}{2}[(1+\frac{1}{\ga})(k+1)+1-\frac{1}{\ga}].
\end{equation*}
\par
Proof of \eqref{eq:suma}.
The proof of \eqref{eq:suma} follows from \eqref{eq:ajsumodd}--\eqref{eq:ajsumeven}.
However, a proof may also be derived independently from \eqref{def:adodd}--\eqref{def:adeven}
and \eqref{def:a} with $j=k$.
Indeed, suppose $k$ is odd.
In view of \eqref{def:adodd} we have
\[
\al_k=2\left(1+\sum_{i<k}\frac{(-1)^i}{\ga^{\sigma(i)}}\ai\right).
\]
Hence, we may write
\begin{align*}
\sum_{i=1}^k\frac{(-1)^i}{\ga^{\sigma(i)}}\ai
=\sum_{i=1}^{k-1}\frac{(-1)^i}{\ga^{\sigma(i)}}\ai+\frac{(-1)^k}{\ga^{\sigma(k)}}\al_k
=\frac{\al_k}{2}-1-\al_k=-1-\frac{\al_k}{2}.
\end{align*}
In view of \eqref{def:a} with $j=k$ we have
\[
1+\frac{\al_k}{2}=1+(1+\frac{1}{\ga})k-\frac{1}{\ga}.
\]
Hence, \eqref{eq:suma} is established for $k$ odd.
\par
Similarly, suppose $k$ is even.
In view of \eqref{def:adeven} we have
\[
\al_k=2\left(1-\ga\sum_{i<k}\frac{(-1)^i}{\ga^{\sigma(i)}}\ai\right).
\]
Hence, we may write
\begin{align*}
\sum_{i=1}^k\frac{(-1)^i}{\ga^{\sigma(i)}}\ai=\sum_{i=1}^{k-1}\frac{(-1)^i}{\ga^{\sigma(i)}}\ai+\frac{(-1)^k}{\ga^{\sigma(k)}}\al_k
=\frac{1}{\ga}-\frac{\al_k}{2\ga}+\frac{\al_k}{\ga}=\frac{1}{\ga}+\frac{\al_k}{2\ga}.
\end{align*}
Now, in view of \eqref{def:a} with $j=k$ we conclude that
\[
\frac{1}{\ga}+\frac{\al_k}{2\ga}=\frac{1}{\ga}+\frac{1}{\ga}[(1+\ga)k-1]=(1+\frac{1}{\ga})k.
\]
The asserted formula~\eqref{eq:suma} follows.
\end{proof}
\begin{rmk}
\label{rmk:ainotincreasing}
Unlike the case $\ga=1$, if $0<\ga<1$ the sequence $\aj$ is not necessarily monotonically increasing
with respect to $j=1,2,\ldots,k$.
\end{rmk}
Indeed, the following holds true.
\medskip
\par\noindent
{\em Claim.}
For all $h\in\mathbb N$ such that $h>(1-\ga^2)^{-1}$ there holds
$\al_{2h}<\al_{2h-1}$.
\par\noindent
Using the explicit values of $\al_{2h},\al_{2h-1}$ as in \eqref{def:a}, we compute:
\[
\frac{1}{2}(\al_{2h}-\al_{2h-1})=(1+\ga)2h-1-(1+\frac{1}{\ga})(2h-1)+\frac{1}{\ga}
=2[(\ga-\frac{1}{\ga})h+\frac{1}{\ga}]
=\frac{2}{\ga}[1-(1-\ga^2)h].
\]
The claim follows.
\subsection*{Properties of the concentration parameters $\dej$}
In this subsection we compute the power decay rates of the concentration parameters $\dej$
as $\rho\to0^+$, $j=1,2,\ldots,k$.
\par
Let $\ka_j=\ka_j(\ga,\tau,h(0),k)>0$, $j=1,2,\ldots,k$ be defined by
\begin{equation*}
\begin{aligned}
&\ka_k=
\begin{cases}
\frac{e^{-4\pi h(0)\Akcal}}{2\al_k^2},
&\mbox{if $k$ is odd;}\\
\frac{\tau\ga e^{4\pi h(0)\Akcal}}{2\al_k^2},
&\mbox{if $k$ is even;}
\end{cases}\\
&\ka_j=
\begin{cases}
\frac{(\tau\ga)^{1/\ga}}{2^{1+1/\ga}\aj^2\al_{j+1}^{2/\ga}},
&\mbox{if $j$ is odd}\\
\frac{\tau\ga}{2^{1+\ga}\aj^2\al_{j+1}^{2\ga}},
&\mbox{if $j$ is even,}\qquad j=1,2,\ldots,k-1.
\end{cases}
\end{aligned}
\end{equation*}
Let $c_j=c_j(\ga,\tau,h(0),k)>0$, $j=1,2,\ldots,k$
be defined by:
\begin{equation*}
c_j=
\begin{cases}
\ka_j\,\ka_{j+1}^{1/\ga}\,\ka_{j+2}\,\ka_{j+3}^{1/\ga}\cdots \ka_k,
&\mbox{if $k$ is odd, $j$ is odd;}\\
\ka_j\,\ka_{j+1}^\ga\,\ka_{j+2}\,\ka_{j+3}^\ga
 \cdots \ka_k^\ga,
&\mbox{if $k$ is odd, $j$ is even;}\\
\ka_j\,\ka_{j+1}^{1/\ga}\,\ka_{j+2}\,\ka_{j+3}^{1/\ga}\cdots \ka_k^{1/\ga},
&\mbox{if $k$ is even, $j$ is odd;}\\
\ka_j\,\ka_{j+1}^\ga\,\ka_{j+2}\,\ka_{j+3}^\ga
 \cdots \ka_k,
&\mbox{if $k$ is even, $j$ is even.}
\end{cases}
\end{equation*}
Let $d_j=d_j(\ga,\tau,h(0),k)>0$ be defined by
\begin{equation}
\label{def:dj}
d_j=c_j^{1/\aj}=
\begin{cases}
c_j^{1/\{2[(1+\frac{1}{\ga})j-\frac{1}{\ga}]\}}&\mbox{if $j$ is odd}\\
c_j^{1/\{2[(1+\ga)j-1]\}}&\mbox{if $j$ is even}.
\end{cases}
\end{equation}
With the above definitions, we have:
\begin{prop}
\label{prop:deimain}
The following power decay rates hold true
for all $j=1,2,\ldots,k$:
\begin{equation}
\begin{aligned}
\label{eq:deltadecay1}
&\dej^{\aj}=c_j\,\rho^{r_j},\ 
\dej=d_j\,\rho^{s_j},&&\mbox{for all }k=1,2,\ldots,k\\
&\frac{\dej}{\de_{j+1}}=\frac{d_j}{d_{j+1}}\,\rho^{\qj},
&&\mbox{for all }k=1,2,\ldots,k-1
\end{aligned}
\end{equation}
where $r_j=r_j(\ga,k)>0$ is defined by
\begin{equation}
\label{def:rj}
r_j=
\begin{cases}
(\ga+1)(k-j)+\ga,&\mbox{$k$ odd, $j$ even}\\
(1+\frac{1}{\ga})(k-j)+1,&\mbox{$k$ odd, $j$ odd}\\
(1+\frac{1}{\ga})(k-j)+\frac{1}{\ga},&\mbox{$k$ even, $j$ odd}\\
(1+\ga)(k-j)+1,&\mbox{$k$ even, $j$ even};
\end{cases}
\end{equation}
$s_j=s_j(\ga,k)>0$ is defined by
\begin{equation}
\label{def:sj}
s_j=
\begin{cases}
\frac{(1+\ga)(k-j)+\ga}{2[(1+\ga)j-1]},&\mbox{if $k$ is odd};\\
\frac{(1+\ga)(k-j)+1}{2[(1+\ga)j-1]},&\mbox{if $k$ is even}.
\end{cases}
\end{equation}
and $q_j=q_j(\ga,k)>0$ is defined by
\begin{equation}
\label{def:qkgi}
\qj=\begin{cases}
\frac{(1+\ga)[(1+\ga)k+\ga-1]}{2[(1+\ga)j+\ga][(1+\ga)j-1]},&\mbox{if $k$ is odd;}\\
\frac{(1+\ga)^2k}{2[(1+\ga)j+\ga][(1+\ga)j-1]},
&\mbox{if $k$ is even}.
\end{cases}
\end{equation}
\end{prop}
In order to prove Proposition~\ref{prop:deimain}, we first establish a recursive formula.
\begin{lemma}
\label{lem:deltarecursive}
We have
\begin{equation}
\label{eq:deltarecursive}
\begin{aligned}
&\de_k^{\al_k}=\ka_k\,\rho;\\
&\dej^{\aj}=
\begin{cases}
\ka_j\,\rho^{1+1/\ga}\left(\delta_{j+1}^{\al_{j+1}}\right)^{1/\ga},
&\mbox{if $j$ is odd}\\
\ka_j\,\rho^{1+\ga}\left(\delta_{j+1}^{\al_{j+1}}\right)^{\ga},
&\mbox{if $j$ is even},
\qquad j=1,2,\ldots,k-1.
\end{cases}
\end{aligned}
\end{equation}
\end{lemma}
\begin{proof}
Suppose $k$ is odd. Then, formula~\eqref{def:adodd} takes the form
\[
\ln(2\al_k^2\de_k^{\al_k})=-4\pi h(0)\Akcal+\ln\rho.
\]
Recalling the explicit value of $\Akcal$ as in \eqref{eq:suma} and 
of $\al_k$ as in \eqref{def:adodd}, we conclude the proof.
\par
Similarly, suppose $k$ is even.
Then, formula~\eqref{def:adeven} takes the form
\[
\ln(2\al_k^2\de_k^{\al_k})=4\pi\ga h(0)\Akcal+\ln(\rho\tau\ga).
\]
Recalling the explicit value of $\Akcal$ as in \eqref{eq:suma} and 
of $\al_k$ as in \eqref{def:adeven}, we conclude the proof.
\par
Suppose $j$ is odd, $j\le k-1$.
Using \eqref{def:adodd} and observing that $j+1$ is even we have
\begin{align*}
\ln(2\aj^2\dej^{\aj})=&2\sum_{i>j}\frac{(-1)^i}{\ga^{\si}}\ln\dei^{\ai}
-4\pi h(0)\Akcal+\ln\rho\\
=&\frac{2}{\ga}\ln\de_{j+1}^{\al_{j+1}}+2\sum_{i>j+1}\frac{(-1)^i}{\ga^{\si}}\ln\dei^{\ai}
-4\pi h(0)\Akcal+\ln\rho,
\end{align*}
where $\Akcal$ is defined in \eqref{def:Ak}.
On the other hand, using \eqref{def:adeven} we have
\begin{equation*}
2\sum_{i>j+1}\frac{(-1)^i}{\ga^{\si}}\ln\dei^{\ai}
=-\frac{1}{\ga}\ln(2\al_{j+1}^2\de_{j+1}^{\al_{j+1}})+4\pi h(0)\Akcal
+\frac{1}{\ga}\ln(\rho\tau\ga).
\end{equation*}
We deduce that
\[
\ln(2\aj^2\dej^{\aj})=\frac{2}{\ga}\ln\de_{j+1}^{\al_{j+1}}-\frac{1}{\ga}\ln(2\al_{j+1}^2\de_{j+1}^{\al_{j+1}})
+\frac{1}{\ga}\ln(\rho\tau\ga)+\ln\rho
\]
and consequently
\[
2\aj^2\dei^{\aj}=\frac{(\tau\ga)^{1/\ga}}{(2\al_{j+1}^2)^{1/\ga}}\rho^{1+1/\ga}\left(\de_{j+1}^{\al_{j+1}}\right)^{1/\ga}.
\]
Hence, the asserted recursive formula follows for $j$ odd.
\par
Similarly, suppose that $j$ is even, $j\le k-1$.
In view of \eqref{def:adeven} and observing that $j+1$ is odd, we have
\begin{align*}
\ln(2\aj^2\dej^{\aj})=&-2\ga\sum_{i>j}\frac{(-1)^i}{\ga^{\si}}\ln\dei^{\ai}+4\pi\ga h(0)\Akcal+\ln(\rho\tau\ga)\\
=&2\ga\ln\de_{j+1}^{\al_{j+1}}-2\ga\sum_{i>j+1}\frac{(-1)^i}{\ga^{\si}}\ln\dei^{\ai}+4\pi\ga h(0)\Akcal+\ln(\rho\tau\ga).
\end{align*}
Since $j+1$ is odd, we have from \eqref{def:adodd}
\[
2\sum_{i>j+1}\frac{(-1)^i}{\ga^{\si}}\ln\dei^{\ai}=\ln(2\al_{j+1}^2\de_{j+1}^{\al_{j+1}})
+4\pi h(0)\Akcal-\ln\rho.
\]
We deduce that
\[
\ln(2\aj^2\dej^{\aj})=2\ga\ln\de_{j+1}^{\al_{j+1}}-\ga\ln(2\al_{j+1}^2\de_{j+1}^{\al_{j+1}})
+\ga\ln\rho+\ln(\rho\tau\ga)
\]
and finally
\[
2\aj^2\dej^{\aj}=\frac{\tau\ga}{(2\al_{j+1}^2)^\ga}\rho^{1+\ga}\left(\de_{j+1}^{\al_{j+1}}\right)^\ga.
\]
The asserted recursive formula is now completely established.
\end{proof}
\begin{proof}[Proof of Proposition~\ref{prop:deimain}]
Proof of the first decay rate in \eqref{eq:deltadecay1}.
We equivalently show that
\begin{equation}
\label{eq:deltadecay}
\de_{k-j}^{\al_{k-j}}=
\begin{cases}
c_{k-j}\,\rho^{(\ga+1)j+\ga},&\mbox{$k$ odd, $j$ odd}\\
c_{k-j}\,\rho^{(1+\frac{1}{\ga})j+1},&\mbox{$k$ odd, $j$ even}\\
c_{k-j}\,\rho^{(1+\frac{1}{\ga})j+\frac{1}{\ga}},&\mbox{$k$ even, $j$ odd}\\
c_{k-j}\,\rho^{(1+\ga)j+1},&\mbox{$k$ even, $j$ even},
\end{cases}
\end{equation}
where $d_{k-j}$ is defined in \eqref{def:dj}.
\par
We argue by induction.
Suppose $k$ is odd.
For $j=0$ we have $\de_k^{\al_k}=\ka_k\,\rho=c_k\,\rho$ and the formula holds true
in this case.
For $j=1$ we have, since $k-1$ is even,
\[
\de_{k-1}^{\al_{k-1}}=\ka_{k-1}\,\rho^{1+\ga}(\de_k^{\al_k})^\ga
=\ka_{k-1}\,\rho^{1+\ga}(\ka_k\,\rho)^\ga=\ka_{k-1}\ka_k^\ga\rho^{(1+\ga)+\ga}
=c_{k-1}\,\rho^{(1+\ga)+\ga},
\]
and the formula is verified for $j=1$ as well.
\par
Hence, assume that the formula is true for all $i\le j<k$ with $j$ even.
Then, $k-j-1$ is even,
\[
\begin{aligned}
\de_{k-j-1}^{\al_{k-j-1}}=&\ka_{k-j-1}\rho^{1+\ga}(\de_{k-j}^{\al_{k-j}})^\ga
=\ka_{k-j-1}\rho^{1+\ga}(\ka_{k-j}\ka_{k-j+1}^{1/\ga}\cdots \ka_k\rho^{(1+1/\ga)j+1})^\ga\\
=&\ka_{k-j-1}\ka_{k-j}^\ga \ka_{k-j+1}\cdots \ka_k^\ga\,\rho^{(1+\ga)(j+1)+\ga}
=c_{k-j-1}\,\rho^{(1+\ga)(j+1)+\ga},
\end{aligned}
\]
and the formula holds true in this case.
\par
Suppose the formula holds true for $j$ odd.
The, $k-j-1$ is odd and we have:
\[
\begin{aligned}
\de_{k-j-1}^{\al_{k-j-1}}=&\ka_{k-j-1}\rho^{1+1/\ga}(\de_{k-j}^{\al_{k-j}})^{1/\ga}
=\ka_{k-j-1}\rho^{1+1/\ga}(\ka_{k-j}\ka_{k-j+1}^\ga \ka_{k-j+2}\cdots \ka_k^\ga\,\rho^{(1+\ga)j+\ga})^{1/\ga}\\
=&\ka_{k-j-1}\ka_{k-j}^{1/\ga}\ka_{k-j+1}\ka_{k-j+2}^{1/\ga}\cdots \ka_k\,\rho^{(1+1/\ga)(j+1)+1}
=c_{k-j-1}\,\rho^{(1+1/\ga)(j+1)+1}.
\end{aligned}
\]
By induction, the formula is established for $k$ odd.
\par
Now, assume that $k$ is even.
For $j=0$ we have $\de_k^{\al_k}=\ka_k\rho=c_k\rho$, and the statement is verified.
For $j=1$ we have
\[
\de_{k-1}^{\al_{k-1}}=\ka_{k-1}\rho^{1+1/\ga}(\de_k^{\al_k})^{1/\ga}
=\ka_{k-1}\rho^{1+1/\ga}(\ka_k\rho)^{1/\ga}
=\ka_{k-1}\ka_k^{1/\ga}\,\rho^{(1+1/\ga)+1/\ga}
=c_{k-1}\,\rho^{(1+1/\ga)+1/\ga},
\]
and the statement is verified.
\par
Hence, suppose the statement holds true for all $i\le j$, $j$ even.
The, $j+1$ is odd, $k-j-1$ is odd.
We have:
\[
\begin{aligned}
\de_{k-j-1}^{\al_{k-j-1}}=&\ka_{k-j-1}\rho^{1+1/\ga}(\de_{k-j}^{\al_{k-j}})^{1/\ga}
=\ka_{k-j-1}\rho^{1+1/\ga}(\ka_{k-j}\ka_{k-j+1}^\ga \ka_{k-j+2}\cdots \ka_k\,\rho^{(1+\ga)j+1})^{1/\ga}\\
=&\ka_{k-j-1}\ka_{k-j}^{1/\ga}\ka_{k-j+1}\ka_{k-j+2}^{1/\ga}\cdots \ka_k^{1/\ga}\,\rho^{(1+\ga)(j+1)+1/\ga}
=c_{k-j-1}\,\rho^{(1+\ga)(j+1)+1/\ga},
\end{aligned}
\]
and the asserted formula follows.
\par
Finally, suppose the statement holds true for all $i\le j$, $j$ odd.
Then, $k-j-1$ is even.
We compute:
\[
\begin{aligned}
\de_{k-j-1}^{\al_{k-j-1}}=&\ka_{k-j-1}\,\rho^{1+\ga}(\de_{k-j}^{\al_{k-j}})^\ga
=\ka_{k-j-1}\,\rho^{1+\ga}(c_{k-j}\rho^{(1+1/\ga)j+1/\ga})^\ga\\
=&\ka_{k-j-1}\ka_{k-j}^\ga \ka_{k-j+1}\cdots \ka_k\,\rho^{(1+\ga)(j+1)+1}
=c_{k-j-1}\,\rho^{(1+\ga)(j+1)+1}.
\end{aligned}
\]
\par
Proof of the second decay rate in \eqref{eq:deltadecay1}.
Using \eqref{eq:deltadecay1}, if $k$ is odd and $j$ is even, we have
\[
\dej^{\aj}=c_j\,\rho^{(1+\ga)(k-j)+\ga}.
\]
Recalling the explicit value of $\aj$ and the definition od $d_j$, we derive
\[
\dej=d_j\rho^{\frac{(1+\ga)(k-j)+\ga}{2[(1+\ga)j-1]}}.
\]
If $k$ is odd and $j$ is odd, we have
\[
\dej^{\aj}=c_j\rho^{(1+\frac{1}{\ga})(k-j)+1}.
\]
Consequently,
\[
\dej=d_j\rho^{\frac{(1+\frac{1}{\ga})(k-j)+1}{2[(1+\frac{1}{\ga})j-\frac{1}{\ga}]}}
=\rho^{\frac{(1+\ga)(k-j)+\ga}{2[(1+\ga)j-1]}},
\]
and the statement follows for $k$ odd.
\par
If $k$ is even and $j$ is odd, we have
\[
\dej^{\aj}=c_j\rho^{(1+\frac{1}{\ga})(k-j)+\frac{1}{\ga}}
\]
and therefore
\[
\dej=d_j\rho^{\frac{(1+\frac{1}{\ga})(k-j)+\frac{1}{\ga}}{2[(1+\frac{1}{\ga})j-\frac{1}{\ga}]}}
=\rho^{\frac{(1+\ga)(k-j)+1}{2[(1+\ga)j-1]}}.
\]
Finally, if $k$ is even and $j$ is even, we have
\[
\dej^\aj=c_j\rho^{(1+\ga)(k-j)+1}.
\]
It follows that
\[
\dei=d_j\rho^{\frac{(1+\ga)(k-j)+1}{2[(1+\ga)j-1]}}.
\]
The proof of the third decay rate in \eqref{eq:deltadecay1}
is an elementary computation; for the reader's convenience we outline it in the Appendix.
\end{proof}
\section{The \lq\lq error function" $\Thj$ \\
(estimation of $\mathcal E_\pm^1$)}
\label{sec:Thetaj}
In this section we justify the choice of \eqref{def:adodd}--\eqref{def:adeven} for the parameters $\ai,\dei$.
\par
We recall that
the \textit{shrinking annuli} $\Aj$ are defined by
\begin{equation*}
\label{def:Aj}
\Aj=\{x\in\Omega:\ \sqrt{\de_{j-1}\de_j}\le|x|<\sqrt{\de_j\de_{j+1}}\},
\qquad j=1,2,\ldots,k,
\end{equation*}
where we set $\de_0=0$ and $\de_{k+1}=+\infty$.
With this definition, for every $i,j=1,2,\ldots,k$ we have
\begin{equation*}
\frac{\Aj}{\dei}=\left\{x\in\frac{\Omega}{\dei}:\ \sqrt{\frac{\de_{j-1}\de_j}{\dei}}\le|x|<\sqrt{\frac{\de_j\de_{j+1}}{\dei}}\right\},
\end{equation*}
and it is readily checked that:
\begin{equation*}
\Aj/\dei\quad
\left\{
\begin{aligned}
&\mbox{runs off to infinity,} &&\mbox{if\ } i<j\\
&\mbox{invades whole space,}&&\mbox{if\ }i=j\\
&\mbox{shrinks to the origin,}&&\mbox{if\ }i>j.
\end{aligned}
\right.
\end{equation*}
We define the \lq\lq error functions" $\Thj$ in $\Aj/\dej$ by setting
\begin{equation*}
\Thj(\frac{x}{\dej})=
\begin{cases}
\Wr(x)-(\aj-2)\ln|x|-\wj(x)+\ln\rho,
&\mbox{if\ }j\ \mbox{is odd}\\
-\ga\Wr(x)-(\aj-2)\ln|x|-\wj(x)+\ln(\rho\tau\ga),
&\mbox{if\ }j\ \mbox{is even}.
\end{cases}
\end{equation*}
Then, we may write
\begin{equation}
\label{eq:Thjmasses}
\begin{cases}
(\rho e^{\Wr(x)}-|x|^{\aj-2}e^{\wj})\chi_{\Aj}=|x|^{\aj-2}e^{\wj}(e^{\Thj(x/\dej)}-1)
&\mbox{if $j$ is odd;}\\
(\ga\rho\tau e^{-\ga\Wr(x)}-|x|^{\aj-2}e^{\wj})\chi_{\Aj}=|x|^{\aj-2}e^{\wj}(e^{\Thj(x/\dej)}-1)
&\mbox{if $j$ is even,}
\end{cases}
\end{equation}
and consequently
\begin{equation*}
\begin{aligned}
\Ep^1=&\sum_{\stackrel{1\le j\le k}{j\ \mathrm{odd}}}|x|^{\aj-2}e^{\wj}(e^{\Thj(x/\dej)}-1)\chi_{\Aj}\\
\Em^1=&\sum_{\stackrel{1\le j\le k}{j\ \mathrm{even}}}|x|^{\aj-2}e^{\wj}(e^{\Thj(x/\dej)}-1)\chi_{\Aj}
\end{aligned}
\end{equation*}
The key point is that $\Thj$ is well estimated in the expanding annulus~$\Aj/\dej$. 
\begin{prop}
\label{prop:Thjmain}
For every fixed $j=1,2,\ldots,k$,
the error term $\Thj$ satisfies:
\begin{equation}
\label{eq:Thjmain}
|\Thj(y)|=O(\dej|y|)+O(\rho^{\bgk})\qquad\mbox{for all }y\in\Aj/\dej,
\end{equation}
with
\begin{equation}
\label{def:bgk}
\bgk=\min\{r_i,q_i,\ i=1,2,\ldots,k\}>0,
\end{equation}
where $r_i,q_i$ are defined in \eqref{def:rj}--\eqref{def:qkgi}.
In particular,
\begin{equation}
\label{eq:Thjcorollary}
\sup_{y\in\Aj/\dej}|\Thj(y)|=O(1).
\end{equation}
\end{prop}
Proposition~\ref{prop:Thjmain} readily implies the following $L^p$-estimates.
\begin{cor}
\label{cor:Thj}
We have, for any $p\ge1$:
\begin{equation*}
\begin{aligned}
&\|\rho e^{\Wr(x)}-|x|^{\aj-2}e^{\wj}\|_{L^p(\Aj)}
=O(\rho^{\bgk-2s_j\frac{p-1}{p}})
&&\mbox{if $j$ is odd}\\
&\|\rho\tau\ga e^{-\ga\Wr(x)}-|x|^{\aj-2}e^{\wj}\|_{L^p(\Aj)}
=O(\rho^{\bgk-2s_j\frac{p-1}{p}}),
&&\mbox{if $j$ is even},
\end{aligned}
\end{equation*}
where $s_j>0$ is defined in \eqref{def:sj} and $\bgk-2s_j\frac{p-1}{p}>0$ for $0\le p-1\ll1$.
\par
In particular, we have
\begin{equation*}
\|\Ep^1\|_{L^p(\Omega)}+\|\Em^1\|_{L^p(\Omega)}=O(\rho^{\bgk-2s_1\frac{p-1}{p}}).
\end{equation*}
\end{cor}
We devote the remaining part of this section to the proof of 
Proposition~\ref{prop:Thjmain} and of Corollary~\ref{cor:Thj}.
We note that we may write:
\begin{equation}
\label{def:Thetaj}
\Thj(\frac{x}{\dej})=
\begin{cases}
P\wj(x)-\wj(x)-(\aj-2)\ln|x|+\ln\rho
+\sum_{i\neq j}\frac{(-1)^{i-1}}{\ga^{\sigma(i)}}P\wi(x),
&\mbox{if\ }j\ \mbox{is odd}\\
P\wj(x)-\wj(x)-(\aj-2)\ln|x|+\ln(\rho\tau\ga)
-\ga\sum_{i\neq j}\frac{(-1)^{i-1}}{\ga^{\sigma(i)}}P\wi(x),
&\mbox{if\ }j\ \mbox{is even}.
\end{cases}
\end{equation}
We recall the expansion of $P\wda$.
\begin{lemma}
\label{lem:proj}
For every $\al\ge2$, $\de>0$ there holds:
\begin{equation}
\label{eq:proj}
\begin{split}
P\wda(x)=&\wda(x)-\ln(2\al^2\de^\al)+4\pi\al H(x,0)+O(\de^\al)\\
=&-2\ln(\de^\al+|x|^\al)+4\pi\al H(x,0)+O(\de^\al).
\end{split}
\end{equation}
\end{lemma}
\begin{proof}
The proof is a direct consequence of the maximum principle,
see, e.g., \cite{GrossiPistoia}.
\end{proof}
The following estimates are a key ingredient.
\begin{lemma}
\label{lem:Aj}
Let $y\in\Aj/\dej$, $j=1,2,\ldots,k$. There holds:
\begin{equation*}
\begin{aligned}
&\frac{1}{|y|^{\ai}}\left(\frac{\dei}{\dej}\right)^{\ai}
=O(\rho^{q_{j-1}}),
&&\mbox{if $i<j$}\\
&|y|^{\ai}\left(\frac{\dej}{\dei}\right)^{\ai}
=O(\rho^{q_j}),
&&\mbox{if $i>j$},
\end{aligned}
\end{equation*} 
where $q_{j-1},q_j>0$ are the constants defined in \eqref{def:qkgi}.
\end{lemma}
\begin{proof}
We have, by definition of $\Aj$, $\sqrt{\de_{j-1}/\dej}\le|y|<\sqrt{\de_{j+1}/\dej}$,
where $\de_0:=0$ and $\de_{k+1}:=+\infty$. 
For $j\ge2$ and for $i<j$ we estimate, using \eqref{eq:deltadecay1}:
\begin{align*}
\frac{1}{|y|^{\ai}}\left(\frac{\dei}{\dej}\right)^{\ai}
\le&\left(\frac{\dej}{\de_{j-1}}\right)^{\ai/2}\left(\frac{\dei}{\dej}\right)^{\ai}
=\left(\frac{\dei^2}{\de_{j-1}\dej}\right)^{\ai/2}\stackrel{i\le j-1}{\le}\left(\frac{\de_{j-1}}{\dej}\right)^{\ai/2}\\
=&\left(\frac{d_{j-1}}{d_j}\rho^{q_{j-1}}\right)^{\ai/2}\stackrel{\ai\ge2}{=}O(\rho^{q_{j-1}}).
\end{align*}
Similarly, for $j\le k-1$ and for $i>j$ we estimate:
\begin{align}
\label{est:igej}
|y|^{\ai}\left(\frac{\dej}{\dei}\right)^{\ai}\le&\left(\frac{\de_{j+1}}{\dej}\right)^{\ai/2}\left(\frac{\dej}{\dei}\right)^{\ai}
=\left(\frac{\de_{j+1}\dej}{\dei^2}\right)^{\ai/2}\stackrel{i\ge j+1}{\le}\left(\frac{\dej}{\de_{j+1}}\right)^{\ai/2}\\
\nonumber
=&\left(\frac{d_{j}}{d_{j+1}}\rho^{q_j}\right)^{\ai/2}\stackrel{\ai\ge2}{=}O(\rho^{q_j}).
\end{align}
\end{proof}
The following expansion, whose proof is elementary, is useful in view of Lemma~\ref{lem:Aj}.
\begin{lemma}
\label{lem:lowerorder}
Let $y\in\Aj/\dej$. Then,
\begin{equation*}
\ln(\dei^{\ai}+|\dej y|^{\ai})=
\begin{cases}
\ai\ln(\dej|y|)+O(\frac{1}{|y|^{\ai}}(\frac{\dei}{\dej})^{\ai})&i<j\\
\ai\ln\dei+\ln(1+|y|^{\ai})&i=j\\
\ai\ln\dei+O(|y|^{\ai}(\frac{\dej}{\dei})^{\ai})&i>j.
\end{cases}
\end{equation*}
\end{lemma}
Using these facts, together with the definition of $\ai,\dei$, we show the following essential estimate.
\begin{lemma}
\label{lem:Thetaj}
For all $y\in\Aj/\dej$ it holds that
\begin{equation*}
\Thj(y)=O(\dej|y|)+\sum_{i=1}^kO(\dei^{\ai})+\sum_{i<j}O(\frac{1}{|y|^{\ai}}(\frac{\dei}{\dej})^{\ai})
+\sum_{i>j}O(|y|^{\ai}(\frac{\dej}{\dei})^{\ai}).
\end{equation*}
\end{lemma}
\begin{proof}
Suppose $j$ is odd.
Then, using the projection expansion \eqref{eq:proj} and Lemma~\ref{lem:lowerorder}, we have
\begin{align*}
\Thj(y)=&P\wj(\dej y)-\wj(\dej y)-(\aj-2)\ln|\dej y|+\ln\rho+\sum_{i\neq j}\frac{(-1)^{i-1}}{\ga^{\sigma(i)}}P\wi(\dej y)\\
=&-\ln(2\aj^2\dej^{\aj})+4\pi\aj h(\dej y)+O(\dej^{\aj})-(\aj-2)\ln|\dej y|+\ln\rho\\
&\quad+\sum_{i\neq j}\frac{(-1)^{i-1}}{\ga^{\sigma(i)}}\{-2\ln(\dei^{\ai}+|\dej y|^{\ai})+4\pi\ai h(\dej y)+O(\dei^{\ai})\}\\
=&-\ln(2\aj^2\dej^{\aj})+4\pi h(0)\sum_{i=1}^k\frac{(-1)^{i-1}}{\ga^{\sigma(i)}}\ai+\ln\rho+\sum_{i=1}^kO(\dei^{\ai})\\
&\quad-(\aj-2)\ln|\dej y|+O(\dej|y|)
-2\sum_{i<j}\frac{(-1)^{i-1}}{\ga^{\sigma(i)}}\left(\ai\ln(\dej|y|)+O\left(\frac{1}{|y|^{\ai}}(\frac{\dei}{\dej})^{\ai}\right)\right)\\
&\quad-2\sum_{i>j}\frac{(-1)^{i-1}}{\ga^{\sigma(i)}}\left(\ln\dei^{\ai}+O\left(|y|^\ai(\frac{\dej}{\dei})^{\ai}\right)\right)\\
=&\underbrace{-(\aj-2)\ln|\dej y|+2\sum_{i<j}\frac{(-1)^{i}}{\ga^{\sigma(i)}}\ai\ln|\dej y|}_{=0, \text{ in view of \eqref{def:adodd}}}\\
&\quad\underbrace{-\ln(2\aj^2\dej^{\aj})-4\pi h(0)\sum_{i=1}^k\frac{(-1)^{i}}{\ga^{\sigma(i)}}\ai+\ln\rho
+2\sum_{i>j}\frac{(-1)^{i}}{\ga^{\sigma(i)}}\ln\dei^{\ai}}_{=0, \text{ in view of  \eqref{def:adodd}}}\\
&\quad+O(\dej|y|)+\sum_{i=1}^kO(\dei^{\ai})+\sum_{i<j}O\left(\frac{1}{|y|^{\ai}}(\frac{\dei}{\dej})^{\ai}\right)
+\sum_{i>j}O\left(|y|^\ai(\frac{\dej}{\dei})^{\ai}\right)\\
=&O(\dej|y|)+\sum_{i=1}^kO(\dei^{\ai})+\sum_{i<j}O\left(\frac{1}{|y|^{\ai}}(\frac{\dei}{\dej})^{\ai}\right)
+\sum_{i>j}O\left(|y|^\ai(\frac{\dej}{\dei})^{\ai}\right).
\end{align*}
Now, suppose that $j$ is even.
We have, from \eqref{def:Thetaj}:
\begin{align*}
\Thj(y)=&P\wj(\dej y)-\wj(\dej y)-(\aj-2)\ln|\dej y|+\ln(\rho\tau\ga)
-\ga\sum_{i\neq j}\frac{(-1)^{i-1}}{\ga^{\sigma(i)}}P\wi(\dej y)\\
=&-\ln(2\aj^2\dej^{\aj})+4\pi\aj h(0)+O(|\dej y|)+O(\dej^{\aj})-(\aj-2)\ln|\dej y|+\ln(\rho\tau\ga)\\
&\quad-\ga\sum_{i\neq j}\frac{(-1)^{i-1}}{\ga^{\sigma(i)}}
\left\{-2\ln(\dei^{\ai}+|\dej y|^\ai)+4\pi\ai h(0)+O(\dei^{\ai})\right\}\\
=&\underbrace{-\ln(2\aj^2\dej^{\aj})-4\pi\ga h(0)\sum_{i\neq j}\frac{(-1)^{i-1}}{\ga^{\sigma(i)}}\ai
+2\ga\sum_{i>j}\frac{(-1)^{i-1}}{\ga^{\sigma(i)}}\ln\dei^{\ai}+\ln(\rho\tau\ga)}_{=0, \text{\ in view of \eqref{def:adeven}}}\\
&\quad\underbrace{-(\aj-2)\ln|\dej y|+2\ga\sum_{i<j}
\frac{(-1)^{i-1}}{\ga^{\sigma(i)}}\ai\ln|\dej y|}_{=0, \text{\ in view of \eqref{def:adeven}}}\\
&\quad+O(|\dej y|)+\sum_{i=1}^kO(\dei^{\ai})
+\sum_{i<j}O\left(\frac{1}{|y|^{\ai}}(\frac{\dei}{\dej})^{\ai}\right)+\sum_{i>j}O\left(|y|^\ai(\frac{\dej}{\dei})^{\ai})\right)\\
=&O(\dej|y|)+\sum_{i=1}^kO(\dei^{\ai})+\sum_{i<j}O\left(\frac{1}{|y|^{\ai}}(\frac{\dei}{\dej})^{\ai}\right)
+\sum_{i>j}O\left(|y|^\ai(\frac{\dej}{\dei})^{\ai})\right).
\end{align*}
The asserted expansion is completely established.
\end{proof}
Now we can prove Proposition~\ref{prop:Thjmain}.
\begin{proof}[Proof of Proposition~\ref{prop:Thjmain}]
We recall from Lemma~\ref{lem:Thetaj} that
\begin{equation*}
\Thj(y)=O(\dej|y|)+\sum_{i=1}^kO(\dei^{\ai})+\sum_{i<j}O(\frac{1}{|y|^{\ai}}(\frac{\dei}{\dej})^{\ai})
+\sum_{i>j}O(|y|^{\ai}(\frac{\dej}{\dei})^{\ai}).
\end{equation*}
We observe that in view of \eqref{eq:deltadecay} we have
\begin{equation*}
\dei^{\ai}=O(\rho^\ga),\qquad\mbox{for all }i=1,2,\ldots,k.
\end{equation*}
The remaining terms are estimated using Lemma~\ref{lem:Aj}.
Hence, \eqref{eq:Thjmain} and \eqref{eq:Thjcorollary} are established.
\end{proof}
\begin{proof}[Proof of Corollary \ref{cor:Thj}]
We begin by showing that
\par
\noindent\textit{Claim~1.}
If $j$ is odd:
\begin{equation}
\label{eq:basicodd}
\int_{\Aj}\left|\rho e^{\Wr}-|x|^{\aj-2}e^{\wj}\right|^p\le C\dej^{2(1-p)}\int_{\Aj/\dej}\frac{|y|^{(\aj-2)p}}{(1+|y|^{\aj})^{2p}}|\Thj(y)|^p\,dy.
\end{equation}
If $j$ is even,
\begin{equation}
\label{eq:basiceven}
\int_{\Aj}\left|\rho\tau e^{-\ga\Wr}-\frac{1}{\ga}|x|^{\aj-2}e^{\wj}\right|^p
\le C\frac{\dej^{2(1-p)}}{\ga^p}\int_{\Aj/\dej}\frac{|y|^{(\aj-2)p}}{(1+|y|^{\aj})^{2p}}|\Thj(y)|^p\,dy.
\end{equation}
Suppose $j$ is odd. Then, using \eqref{eq:Thjmasses}
\begin{align*}
\int_{\Aj}\left||x|^{\aj-2}e^{\wj}-\rho e^{\Wr}\right|^p
=&\int_{\Aj}|x|^{p(\aj-2)}e^{p\wj}\left|1-e^{\Thj(x/\dej)}\right|^p\,dx\\
\stackrel{y=x/\dej}{=}&\int_{\Aj/\dej}|\dej y|^{p(\aj-2)}e^{p\wj(\dej y)}|1-e^{\Thj(y)}|^p\dej^2\,dy\\
\le&C\int_{\Aj/\dej}|\dej y|^{p(\aj-2)}\left[\frac{\dej^{\aj}}{(\dej^\aj+|\dej y|^{\aj})^2}\right]^p|\Thj(y)|^p\dej^2\,dy\\
=&C\dej^{2(1-p)}\int_{\Aj/\dej}\frac{|y|^{p(\aj-2)}}{(1+|y|^{\aj})^{2p}}|\Thj(y)|^p\,dy.
\end{align*}
Similarly, if $j$ is even, we compute, using \eqref{eq:Thjmasses}:
\begin{align*}
\int_{\Aj}&\left||x|^{\aj-2}e^{\wj}-\rho\tau\ga e^{-\ga\Wr}\right|^p\,dx
=\int_{\Aj}|x|^{p(\aj-2)}e^{p\wj}\left|1-e^{\Thj(x/\dej)}\right|^p\,dx\\
&\stackrel{y=x/\dej}{=}\int_{\Aj/\dej}|\dej y|^{p(\aj-2)}e^{p\wj(\dej y)}\left|1-e^{\Thj(y)}\right|^p\dej^2\,dy\\
&\le C\dej^{2(1-p)}\int_{\Aj/\dej}\frac{|y|^{p(\aj-2)}}{(1+|y|^{\aj})^{2p}}|\Thj(y)|^p\,dy.
\end{align*}
\noindent
\smallskip\textit{Claim~2.}
The following decay estimates hold true.
\begin{equation*}
\dej^{2(1-p)}\int_{\Aj/\dej}\frac{|y|^{p(\aj-2)}}{(1+|y|^{\aj})^{2p}}|\Thj(y)|^p\,dy\le C\dej^{2(1-p)+p\bgk}
=o(\de_k^{2(1-p)+p\bgk}).
\end{equation*}
Indeed, in view of \eqref{eq:Thjmain} we have
\begin{align*}
\dej^{2(1-p)}&\int_{\Aj/\dej}\frac{|y|^{p(\aj-2)}}{(1+|y|^{\aj})^{2p}}|\Thj(y)|^p\,dy
\le C\dej^{2(1-p)}\int_{\Aj/\dej}\frac{|y|^{p(\aj-2)}}{(1+|y|^{\aj})^{2p}}\left|\dej|y|+\rho^{\bgk}\right|^p\\
&\le C\dej^{2(1-p)}\dej^p\int_{\Aj/\dej}\frac{|y|^{p(\aj-2)}|y|^p}{(1+|y|^{\aj})^{2p}}\,dy
+C\dej^{2(1-p)}\rho^{p\bgk}\int_{\Aj/\dej}\frac{|y|^{p(\aj-2)}}{(1+|y|^{\aj})^{2p}}\,dy.
\end{align*}
Since the integrals appearing above are uniformly bounded, the asserted decay rates follow.
\end{proof}
\section{The error terms $\Rrho$ and $\Srho$\\
(estimation of $\mathcal E_\pm^2,\mathcal E_\pm^3$)}
\label{sec:RrhoSrho}
We recall from Section~\ref{sec:Ansatz} that
\[
f(s):=e^s-\tau e^{-\ga s}
\]
and
\[
\begin{aligned}
\Rrho(x):=&\rho f(W_\rho)+\Delta W_\rho
=\rho f(W_\rho)-\sum_{i=1}^k\frac{(-1)^{i-1}}{\ga^{\sigma(i)}}|x|^{\ai-2}e^{\wi}
=\Ep-\Em\\
\Srho(x):=&\rho f'(\Wr)-\sum_{i=1}^k|x|^{\ai-2}e^{\wi}
=\Ep+\ga\Em,
\end{aligned}
\]
where
\begin{equation*}
\label{def:EpEm}
\begin{aligned}
\Ep:=&\rho e^{\Wr}-\sum_{\stackrel{1\le i\le k}{i\ \mathrm{odd}}}|x|^{\ai-2}e^{\wi}\\
\Em:=&\rho\tau e^{-\ga\Wr}-\frac{1}{\ga}\sum_{\stackrel{1\le i\le k}{i\ \mathrm{even}}}|x|^{\ai-2}e^{\wi}.
\end{aligned}
\end{equation*}
Our aim in this section is to obtain power decay estimates
for $\|\Ep\|_{L^p(\Omega)}$ and $\|\Em\|_{L^p(\Omega)}$,
for $p\ge1$, $p-1\ll1$.
More precisely, we 
establish the following
\begin{prop}
\label{prop:RS}
There exists $p_0>1$ such that for every $p\in[1,p_0)$ there exists $\overline\beta_p=\overline\beta_p(\tau,\ga,k)>0$ such that:
\begin{equation}
\label{est:E}
\|\Ep\|_{L^p(\Omega)}+\|\Em\|_{L^p(\Omega)}=O(\rho^{\overline\beta_p})
\end{equation}
and consequently
\begin{equation*}
\|\Rrho\|_{L^p(\Omega)}+\|\Srho\|_{L^p(\Omega)}=O(\rho^{\overline\beta_p})
\end{equation*}
\end{prop}
By taking $p=1$ in \eqref{est:E} and using \eqref{eq:Liouvillemass} we derive from the above:
\begin{cor}
\label{cor:masses}
For any $r>0$ there holds, as $\rho\to0^+$:
\begin{equation}
\label{est:mass}
\begin{aligned}
&\int_{B_r(0)}\rho e^{\Wr}\,dx=\sum_{\stackrel{1\le j\le k}{j\ \mathrm{odd}}}4\pi\aj+o(1)\\
&\int_{B_r(0)}\tau\rho e^{-\ga\Wr}\,dx=\frac{1}{\ga}\sum_{\stackrel{1\le j\le k}{j\ \mathrm{odd}}}4\pi\aj+o(1).
\end{aligned}
\end{equation}
Moreover, for any $q>1$ we have
\begin{equation}
\label{est:q}
\int_\Omega(\rho e^{\Wr})^q\,dx=O(\rho^{2s_k(1-q)}),
\qquad
\int_\Omega(\tau\rho e^{-\ga\Wr})^q\,dx=O(\rho^{2s_k(1-q)}).
\end{equation}
\end{cor}
In order to prove Proposition~\ref{prop:RS} we recall from Section~\ref{sec:Ansatz} that
$\Ep=\Ep^1+\Ep^2+\Ep^3$, where
\begin{equation*}
\begin{aligned}
\Ep^1=&\sum_{\stackrel{1\le j\le k}{j\ \mathrm{odd}}}\Big(\rho e^{\Wr}-|x|^{\aj-2}e^{\wj}\Big)\chi_{\Aj}\\
\Ep^2=&\sum_{\stackrel{1\le j\le k}{j\ \mathrm{even}}}\rho e^{\Wr}\chi_{\Aj}\\
\Ep^3=&\sum_{\stackrel{1\le j\le k}{j\neq i}}\sum_{\stackrel{1\le i\le k}{i\ \mathrm{odd}}}|x|^{\ai-2}e^{\wi}\chi_{\Aj}
\end{aligned}
\end{equation*}
and $\Em=\Em^1+\Em^2+\Em^3$, where
\begin{equation*}
\begin{aligned}
\Em^1=&\sum_{\stackrel{1\le j\le k}{j\ \mathrm{even}}}\Big(\rho\tau e^{-\ga\Wr}-|x|^{\aj-2}e^{\wj}\Big)\chi_{\Aj}\\
\Em^2=&\sum_{\stackrel{1\le j\le k}{j\ \mathrm{odd}}}\rho\tau e^{-\ga\Wr}\chi_{\Aj}\\
\Em^3=&\frac{1}{\ga}\sum_{\stackrel{1\le j\le k}{j\neq i}}\sum_{\stackrel{1\le i\le k}{i\ \mathrm{even}}}|x|^{\ai-2}e^{\wi}\chi_{\Aj}.
\end{aligned}
\end{equation*}
The errors $\Ep^1,\Em^1$ are already estimated in Corollary~\ref{cor:Thj}.
We estimate $\Ep^2,\Em^2$.
To this end, we first establish the following auxiliary estimates.
\begin{lemma}
\label{lem:I2decay}
If $j$ is odd:
\begin{align*}
\int_{\Aj}|\rho\tau e^{-\ga\Wr}|^p
\le C\rho^{(1+\ga)p}\dej^{2(1+p\ga)}
\left[\left(\frac{\de_{j+1}}{\dej}\right)^{p\ga\frac{\aj+2}{2}+1}
+\left(\frac{\dej}{\de_{j-1}}\right)^{p\ga\frac{\aj-2}{2}-1}\right];
\end{align*}
if $j$ is even,
\begin{align*}
\int_{\Aj}|\rho e^{\Wr}|^p
\le C\rho^{(1+\frac{1}{\ga})p}(\tau\ga)^{\frac{p}{\ga}}\dej^{2(1+\frac{p}{\ga})}
\left[\left(\frac{\de_{j+1}}{\dej}\right)^{\frac{p}{\ga}\frac{\aj+2}{2}+1}
+\left(\frac{\dej}{\de_{j-1}}\right)^{\frac{p}{\ga}\frac{\aj-2}{2}-1}\right],
\end{align*}
where for the sake of simplicity it is understood that if $j=1$ only the first term on the right hand side exists
and if $j=k$ only the second term on the right hand side exists.
\end{lemma}
\begin{proof}
We begin by showing the following.
\par
\textit{Claim~1.}
If $j$ is odd:
\begin{align*}
\int_{\Aj}|\rho\tau e^{-\ga\Wr}|^p\le C\rho^{(1+\ga)p}\dej^{2(1+p\ga)}
\int_{\Aj/\dej}e^{-p\ga\Thj(y)}\frac{(1+|y|^\aj)^{2p\ga}}{|y|^{(\aj-2)p\ga}}\,dy.
\end{align*}
If $j$ is even,
\begin{align*}
\int_{\Aj}|\rho e^{\Wr}|^p\le C\rho^{(1+\frac{1}{\ga})p}(\tau\ga)^{p/\ga}\dej^{2(1+p/\ga)}
\int_{\Aj/\dej}e^{-\frac{p}{\ga}\Thj(y)}\frac{(1+|y|^\aj)^{2p/\ga}}{|y|^{(\aj-2)p/\ga}}\,dy.
\end{align*}
Proof of Claim~1.
Suppose $j$ is odd.
We compute:
\begin{align*}
\int_{\Aj}|\rho\tau e^{-\ga\Wr}|^p
\le C&\rho^p\int_{\Aj}e^{-p\ga\sum_{i=1}^k\frac{(-1)^{i-1}}{\ga^{\si}}P\wi}
=C\rho^p\int_{\Aj}e^{-p\ga\Thj(x/\dej)}e^{p\ga(-\wj-(\aj-2)\ln|x|+\ln\rho)}\\
=&C\rho^{(1+\ga)p}\int_{\Aj}e^{-p\ga\Thj(\frac{x}{\dej})}\left[\frac{(\dej^{\aj}+|x|^{\aj})^2}{\dej^{\aj}}\right]^{p\ga}
\frac{dx}{|x|^{(\aj-2)p\ga}}\\
\stackrel{x=\dej y}{=}&C\rho^{(1+\ga)p}\dej^{2(1+p\ga)}\int_{\Aj/\dej}
e^{-p\ga\Thj(y)}\frac{(1+|y|^{\aj})^{2p\ga}}{|y|^{(\aj-2)p\ga}}\,dy\\
\le C&\rho^{(1+\ga)p}\dej^{2(1+p\ga)}
\int_{\sqrt{\frac{\de_{j-1}}{\dej}}\le|y|<\sqrt{\frac{\de_{j+1}}{\dej}}}
\frac{(1+|y|^{\aj})^{2p\ga}}{|y|^{(\aj-2)p\ga}}\,dy.
\end{align*}
Similarly, suppose $j$ is even.
We compute:
\begin{align*}
\int_{\Aj}|\rho e^{\Wr}|^p
=&\rho^p\int_{\Aj}e^{p\sum_{i=1}^k\frac{(-1)^{i-1}}{\ga^{\si}}P\wi}
=\rho^p\int_{\Aj}e^{-\frac{p}{\ga}\Thj(\frac{x}{\dej})-\frac{p}{\ga}[\wj+(\aj-2)\ln|x|-\ln(\rho\tau\ga)]}\\
=&\rho^p\int_{\Aj}e^{-\frac{p}{\ga}\Thj(\frac{x}{\dej})}\left[\frac{(\dej^{\aj}+|x|^{\aj})^2}{\dej^{\aj}}\right]^{p/\ga}
(\rho\tau\ga)^{p/\ga}\frac{dx}{|x|^{(\aj-2)p/\ga}}\\
\stackrel{x=\dej y}{=}&\rho^{p(1+\frac{1}{\ga})}\dej^{2(1+\frac{p}{\ga})}
\int_{\Aj/\dej}e^{-\frac{p}{\ga}\Thj(y)}\frac{(1+|y|^{\aj})^{2p/\ga}}{|y|^{(\aj-2)p/\ga}}(\tau\ga)^{p/\ga}\,dy.
\end{align*}
Claim~1 is thus established.
\par
\textit{Claim~2.}
For any $\eta>0$ we have:
\begin{equation*}
\int_{\Aj/\dej}\frac{(1+|y|^{\aj})^{2p\eta}}{|y|^{(\aj-2)p\eta}}\,dy
\le C\left[\left(\frac{\de_{j+1}}{\dej}\right)^{p\eta\frac{\aj+2}{2}+1}
+\left(\frac{\dej}{\de_{j-1}}\right)^{p\eta\frac{\aj-2}{2}-1}\right].
\end{equation*}
Proof of Claim~2.
We compute:
\begin{align*}
\int_{\Aj/\dej}\frac{(1+|y|^{\aj})^{2p\eta}}{|y|^{(\aj-2)p\eta}}\,dy
=&\int_{\sqrt{\frac{\de_{j-1}}{\de_{j}}}\le|y|<1}\frac{(1+|y|^{\aj})^{2p\eta}}{|y|^{(\aj-2)p\eta}}\,dy
+\int_{1\le|y|<\sqrt{\frac{\de_{j+1}}{\dej}}}\frac{(1+|y|^{\aj})^{2p\eta}}{|y|^{(\aj-2)p\eta}}\,dy\\
=:&I+II.
\end{align*}
We estimate, for $j\ge2$:
\begin{align*}
I\le C\int_{\sqrt{\frac{\de_{j-1}}{\de_{j}}}}^1\frac{r\,dr}{r^{(\aj-2)p\eta}}
\le C\left[1+\left(\frac{\dej}{\de_{j-1}}\right)^{p\eta\frac{\aj-2}{2}-1}\right].
\end{align*}
Similarly, for $j\le k-1$, we have
\begin{align*}
II\le C\int_1^{\sqrt{\frac{\de_{j+1}}{\de_{j}}}}
|y|^{2p\eta\aj-p\eta(\aj-2)}\,dy=\int_1^{\sqrt{\frac{\de_{j+1}}{\de_{j}}}}
|y|^{p\eta(\aj+2)}\,dy\le C\left(\frac{\de_{j+1}}{\dej}\right)^{p\eta\frac{\aj+2}{2}+1}.
\end{align*}
\end{proof}
\begin{lemma}
\label{lem:powerdecay}
The following power decay rates hold true.
\par
If $j$ is odd:
\begin{equation*}
\begin{aligned}
&\rho^{(1+\ga)p}\dej^{2(1+p\ga)}\left(\frac{\de_{j+1}}{\dej}\right)^{p\ga\frac{\aj+2}{2}+1}
=O\left(\rho^{(1+\ga)\frac{p}{2}}\left(\frac{\dej}{\de_{j+1}}\right)^{1+p\ga}\de_{j+1}^{2+p\ga-p}\right)\\
&\rho^{(1+\ga)p}\dej^{2(1+p\ga)}\left(\frac{\dej}{\de_{j-1}}\right)^{p\ga\frac{\aj-2}{2}-1}
=O\left(\rho^{(1+\ga)\frac{p}{2}}\left(\frac{\de_{j-1}}{\dej}\right)^{1-p}\dej^{2+p\ga-p}\right);
\end{aligned}
\end{equation*}
if $j$ is even:
\begin{equation*}
\begin{aligned}
&\rho^{(1+\frac{1}{\ga})p}\dej^{2(1+\frac{p}{\ga})}
\left(\frac{\de_{j+1}}{\dej}\right)^{\frac{p}{\ga}\frac{\aj+2}{2}+1}
=O\left(\rho^{(1+\frac{1}{\ga})\frac{p}{2}}\left(\frac{\dej}{\de_{j+1}}\right)^{1+p/\ga}\de_{j+1}^{2+\frac{p}{\ga}-p}\right)\\
&\rho^{(1+\frac{1}{\ga})p}\dej^{2(1+p/\ga)}\left(\frac{\dej}{\de_{j-1}}\right)^{\frac{p}{\ga}\frac{\aj-2}{2}-1}
=O\left(\rho^{(1+\frac{1}{\ga})\frac{p}{2}}\left(\frac{\de_{j-1}}{\dej}\right)^{-p+1}\dej^{2+\frac{p}{\ga}-p}\right)
\end{aligned}
\end{equation*}
\end{lemma}
\begin{proof}
Proof of the first decay rate for $j$ odd.
Since $j+1$ is even, in view of the recursive formula~\eqref{eq:arec} we have 
$\ga\aj=\al_{j+1}-2(1+\ga)$ and therefore
\[
p\ga\frac{\aj+2}{2}=\frac{p}{2}(\al_{j+1}-2(1+\ga))+p\ga=\frac{p}{2}\al_{j+1}-p.
\]
We deduce that
\begin{align*}
\rho^{(1+\ga)p}\dej^{2(1+p\ga)}&\left(\frac{\de_{j+1}}{\dej}\right)^{p\ga\frac{\aj+2}{2}+1}\\
=&\rho^{(1+\ga)p}\dej^{2(1+p\ga)}\frac{\de_{j+1}^{\frac{p}{2}\al_{j+1}-p+1}}{\dej^{p\ga\frac{\aj+2}{2}+1}}
=\rho^{(1+\ga)p}\dej^{2(1+p\ga)}\frac{(\de_{j+1}^{\al_{j+1}})^{p/2}\de_{j+1}^{-p+1}}{(\dej^{\aj})^{p\ga/2}\dej^{p\ga+1}}\\
=&\rho^{(1+\ga)p}\dej^{1+p\ga}\frac{(\de_{j+1}^{\al_{j+1}})^{p/2}}{(\dej^{\aj})^{p\ga/2}}\de_{j+1}^{1-p}.
\end{align*}
Recalling from \eqref{eq:deltarecursive} that $\dej^{\aj}=\kappa_j\rho^{1+1/\ga}(\de_{j+1}^{\al_{j+1}})^{1/\ga}$,
in turn we derive that
\begin{align*}
\rho^{(1+\ga)p}\dej^{1+p\ga}
\frac{(\de_{j+1}^{\al_{j+1}})^{p/2}}{[\rho^{1+1/\ga}(\de_{j+1}^{\al_{j+1}})^{1/\ga}]^{p\ga/2}}\de_{j+1}^{1-p}
=\rho^{(1+\ga)p/2}\left(\frac{\dej}{\de_{j+1}}\right)^{1+p\ga}\de_{j+1}^{1+p\ga-p+1},
\end{align*}
and the asserted estimate follows.
\par
Proof of the second decay rate for $j$ odd.
In view of the recursive formula~\eqref{eq:arec}, we have $\aj=\al_{j-1}/\ga+2(1+1/\ga)$ and therefore
\[
p\ga\frac{\aj}{2}=\frac{p}{2}\al_{j-1}+p(\ga+1).
\]
It follows that
\begin{align*}
\rho^{(1+\ga)p}\dej^{2(1+p\ga)}&\left(\frac{\dej}{\de_{j-1}}\right)^{p\ga\frac{\aj-2}{2}-1}
=\rho^{(1+\ga)p}\dej^{2(1+p\ga)}\frac{(\dej^{\aj})^{p\ga/2}\dej^{-p\ga-1}}{\de_{j-1}^{p\ga\frac{\aj-2}{2}-1}}\\
=&\rho^{(1+\ga)p}\dej^{2(1+p\ga)}\frac{(\dej^{\aj})^{p\ga/2}\dej^{-p\ga-1}}{\de_{j-1}^{\frac{p}{2}\al_{j-1}+p(\ga+1)-p\ga-1}}
=\rho^{(1+\ga)p}\dej^{1+p\ga}\frac{(\dej^{\aj})^{p\ga/2}}{(\de_{j-1}^{\al_{j-1}})^{p/2}\de_{j-1}^{p-1}}.
\end{align*}
Since $j-1$ is even, in view of the recursive formula~\eqref{eq:deltarecursive} that
$\de_{j-1}^{\al_{j-1}}=\kappa_{j-1}\rho^{1+\ga}(\dej^{\aj})^\ga$.
We deduce that
\begin{align*}
\rho^{(1+\ga)p}\dej^{1+p\ga}
\frac{(\dej^{\aj})^{p\ga/2}}{[\rho^{1+\ga}(\dej^{\aj})^\ga]^{p/2}\de_{j-1}^{p-1}}
=\rho^{(1+\ga)p/2}\dej^{1+p\ga}\de_{j-1}^{1-p}
=\rho^{(1+\ga)p}\dej^{1+p\ga-p+1}\left(\frac{\de_{j-1}}{\dej}\right)^{1-p},
\end{align*}
and the asserted estimate follows.
\par
Proof of the first estimate for $j$ is even.
Since $j+1$ is odd, in view of the recursive formula~\eqref{eq:arec}, we have:
$\aj/\ga=\al_{j+1}-2(1+1/\ga)$.
Hence, we may write
\begin{align*}
\de_{j+1}^{\frac{p}{\ga}\frac{\aj+2}{2}+1}=\de_{j+1}^{\frac{p}{2}(\al_{j+1}-2(1+1/\ga))+p/\ga+1}
=(\de_{j+1}^{\al_{j+1}})^{p/2}\de_{j+1}^{-p+1}
\end{align*}
Hence,
\begin{align*}
\left(\frac{\de_{j+1}}{\dej}\right)^{\frac{p}{\ga}\frac{\aj+2}{2}+1}
=\frac{(\de_{j+1}^{\al_{j+1}})^{p/2}\de_{j+1}^{-p+1}}{(\dej^{\aj})^{\frac{p}{2\ga}}\dej^{\frac{p}{\ga}+1}}
\end{align*}
Since $j$ is even, in view of the recursive formula~ \eqref{eq:deltarecursive} we have that 
$\dej^{\aj}=\kappa_j\rho^{1+\ga}(\de_{j+1}^{\al_{j+1}})^\ga$.
We deduce that
\begin{align*}
\rho^{(1+\frac{1}{\ga})p}\dej^{2(1+\frac{p}{\ga})}
\left(\frac{\de_{j+1}}{\dej}\right)^{\frac{p}{\ga}\frac{\aj+2}{2}+1}=&\rho^{(1+\frac{1}{\ga})p}\dej^{2(1+\frac{p}{\ga})}
\frac{(\de_{j+1}^{\al_{j+1}})^{p/2}\de_{j+1}^{-p+1}}{[\rho^{1+\ga}(\de_{j+1}^{\al_{j+1}})^\ga]^{\frac{p}{{2\ga}}}\dej^{\frac{p}{\ga}+1}}
=\rho^{(1+\frac{1}{\ga})\frac{p}{2}}\dej^{1+\frac{p}{\ga}}\de_{j+1}^{-p+1}\\
=&\rho^{(1+\frac{1}{\ga})\frac{p}{2}}\left(\frac{\dej}{\de_{j+1}}\right)^{1+p/\ga}\de_{j+1}^{1+\frac{p}{\ga}-p+1},
\end{align*}
as desired.
\par
Proof of the second estimate for $j$ even.
Since $j$ is even, in view of the recursive formula \eqref{eq:arec} that $\aj/\ga=\al_{j-1}+2(1+1/\ga)$
and therefore
\[
\de_{j-1}^{\frac{p}{\ga}\frac{\aj-2}{2}-1}= \de_{j-1}^{\frac{p}{2}(\al_{j-1}+2(1+\frac{1}{\ga}))-\frac{p}{\ga}-1}
=(\de_{j-1}^{\al_{j-1}})^{p/2}\de_{j-1}^{p-1}.
\]
We deduce that
\[
\left(\frac{\dej}{\de_{j-1}}\right)^{\frac{p}{\ga}\frac{\aj-2}{2}-1}=
\frac{(\dej^{\aj})^{\frac{p}{2\ga}}\dej^{-\frac{p}{\ga}-1}}{(\de_{j-1}^{\al_{j-1}})^{p/2}\de_{j-1}^{p-1}}
\]
Since $j-1$ is odd, in view of \eqref{eq:deltarecursive} we have:
\[
\de_{j-1}^{\al_{j-1}}=\kappa_{j-1}\rho^{1+\frac{1}{\ga}}(\dej^\aj)^{1/\ga}.
\]
It follows that
\[
\left(\frac{\dej}{\de_{j-1}}\right)^{\frac{p}{\ga}\frac{\aj-2}{2}-1}
\le C
\frac{(\dej^{\aj})^{\frac{p}{2\ga}}\dej^{-\frac{p}{\ga}-1}}{(\rho^{1+\frac{1}{\ga}}(\dej^\aj)^{1/\ga})^{p/2}\de_{j-1}^{p-1}}
=C\frac{\dej^{-\frac{p}{\ga}-1}}{\rho^{\frac{p}{2}(1+\frac{1}{\ga})}\de_{j-1}^{p-1}}.
\]
Consequently,
\begin{align*}
\rho^{(1+\frac{1}{\ga})p}\dej^{2(1+p/\ga)}\left(\frac{\dej}{\de_{j-1}}\right)^{\frac{p}{\ga}\frac{\aj-2}{2}-1}
=&\rho^{(1+\frac{1}{\ga})p}\dej^{2(1+p/\ga)}\frac{\dej^{-\frac{p}{\ga}-1}}{\rho^{\frac{p}{2}(1+\frac{1}{\ga})}\de_{j-1}^{p-1}}
=\rho^{\frac{p}{2}(1+\frac{1}{\ga})}\dej^{1+\frac{p}{\ga}}\de_{j-1}^{-p+1}\\
=&\rho^{\frac{p}{2}(1+\frac{1}{\ga})}\left(\frac{\de_{j-1}}{\dej}\right)^{-p+1}\dej^{1+\frac{p}{\ga}-p+1},
\end{align*}
as desired.
The asserted decay estimates are completely established.
\end{proof}
\begin{lemma}
\label{lem:I3}
There holds:
\begin{equation*}
\int_{\Aj}\left||x|^{\ai-2}e^{\wi}\right|^p\,dx
=
\begin{cases}
O\left(\dei^{2(1-p)}\left(\frac{\dej}{\de_{j+1}}\right)^{p\frac{\ai-2}{2}+1}\right),&\mbox{if\ $i>j$}\\
O\left(\dei^{2(1-p)}\left(\frac{\de_{j-1}}{\dej}\right)^{p\frac{\ai}{2}+p+1}\right),&\mbox{if\ $i<j$}.
\end{cases}
\end{equation*}
\end{lemma}
\begin{proof}
We have:
\begin{align*}
\int_{\Aj}&
|x|^{p(\ai-2)}\left[\frac{\dei^{\ai}}{(\dei^{\ai}+|x|^{\ai})^2}\right]^p\,dx
=\dei^{p\ai}\int_{\Aj}\frac{|x|^{p(\ai-2)}}{(\dei^{\ai}+|x|^{\ai})^{2p}}\,dx\\
&\stackrel{x=\dei y}{=}
\dei^{2(1-p)}\int_{\Aj/\dei}\frac{|y|^{p(\ai-2)}}{(1+|y|^{\ai})^{2p}}\,dy.
\end{align*}
Suppose $i>j$ (i.e., $i\ge j+1$).
Then, $\sqrt{\dej\de_{j+1}}/\dei=o(\sqrt{\dej/\de_{j+1}})=o(1)$ as $\rho\to0^+$
and therefore
\begin{align*}
\int_{\Aj/\dei}\frac{|y|^{p(\ai-2)}}{(1+|y|^{\ai})^{2p}}\,dy
\le C\int_{\frac{\sqrt{\de_{j-1}\dej}}{\dei}}^{\frac{\sqrt{\dej\de_{j+1}}}{\dei}}
r^{p(\ai-2)}\,rdr
\le C\left(\frac{\dej}{\de_{j+1}}\right)^{p\frac{\ai-2}{2}+1}.
\end{align*}
Suppose $i<j$ (i.e., $i\le j-1$). Then, 
$\sqrt{\de_{j-1}\dej}/\dei\ge C^{-1}\sqrt{\dej/\de_{j-1}}\to+\infty$ as $\rho\to0^+$
and therefore
\begin{align*}
\int_{\Aj/\dei}\frac{|y|^{p(\ai-2)}}{(1+|y|^{\ai})^{2p}}\,dy
\le C\int_{\frac{\sqrt{\de_{j-1}\dej}}{\dei}}^{\frac{\sqrt{\dej\de_{j+1}}}{\dei}}
r^{p(\ai-2)-2p\ai}\,rdr
\le C\left(\frac{\de_{j-1}}{\dej}\right)^{p\frac{\ai}{2}+p+1}.
\end{align*}
The asserted decay estimates are thus established.
\end{proof}
Now we can provide the proof of Proposition~\ref{prop:RS}.
\begin{proof}[Proof of Proposition~\ref{prop:RS}]
The proof is a direct consequence of Lemma~\ref{lem:I2decay},
Lemma~\ref{lem:powerdecay} and Lemma~\ref{lem:I3}.
\end{proof}
\begin{proof}[Proof of Corollary~\ref{cor:masses}]
We  decompose:
\begin{equation*}
\begin{aligned}
\int_{B_r}\rho e^{\Wr}=&\sum_{j=1}^k\int_{B_r\cap\Aj}\rho e^{\Wr}\\
=&\sum_{\stackrel{j=1}{j\ \mathrm{odd}}}^k\int_{B_r\cap\Aj}|x|^{\aj-2}e^{\wj}
+\sum_{\stackrel{j=1}{j\ \mathrm{odd}}}^k\int_{B_r\cap\Aj}(\rho e^{\Wr}-|x|^{\aj-2}e^{\wj})
+\sum_{\stackrel{j=1}{j\ \mathrm{even}}}^k\int_{B_r\cap\Aj}\rho e^{\Wr}\\
=&\sum_{\stackrel{j=1}{j\ \mathrm{odd}}}^k\int_{B_r\cap\Aj}|x|^{\aj-2}e^{\wj}
+\int_{B_r}\Ep^1+\int_{B_r}\Ep^2.
\end{aligned}
\end{equation*}
Using \eqref{est:E} with $p=1$, we obtain
\begin{equation*}
\int_{B_r}\rho e^{\Wr}=\sum_{\stackrel{j=1}{j\ \mathrm{odd}}}^k\int_{B_r\cap\Aj}|x|^{\aj-2}e^{\wj}
+O(\rho^{\overline\beta_1})
\end{equation*}
for some $\overline\beta_1>0$.
On the other hand, by a standard rescaling and \eqref{eq:Liouvillemass}, 
\[
\int_{B_r\cap\Aj}|x|^{\aj-2}e^{\wj}=4\pi\aj+o(1).
\]
Hence, \eqref{est:mass} follows.
\par
Proof of \eqref{est:q}.
We have:
\begin{equation*}
\begin{aligned}
\int_\Omega&(\rho e^{\Wr})^q=\sum_{j=1}^k\int_{\Aj}(\rho e^{\Wr}\chi_{\Aj})^q\\
\le&C\sum_{\stackrel{j=1}{j\ \mathrm{odd}}}^k\int_{\Aj}(\rho e^{\Wr}-|x|^{\aj-2}e^{\wj})^q
+C\sum_{\stackrel{j=1}{j\ \mathrm{odd}}}^k\int_{\Aj}|x|^{q(\aj-2)}e^{q\wj}
+C\sum_{\stackrel{j=1}{j\ \mathrm{even}}}^k\int_{\Aj}(\rho e^{\Wr})^q\\
=&C\int_\Omega|\Ep^1|^q+C\sum_{\stackrel{j=1}{j\ \mathrm{odd}}}^k\int_{\Aj}\frac{|x|^{q(\aj-2)}\dej^{q\aj}}{(\dej^{\aj}+|x|^{\aj})^{2q}}
+C\int_\Omega|\Ep^2|^q.
\end{aligned}
\end{equation*} 
By rescaling we find
\[
\int_{\Aj}\frac{|x|^{q(\aj-2)}\dej^{q\aj}}{(\dej^{\aj}+|x|^{\aj})^{2q}}=O(\dej^{2(1-q)})=O(\rho^{2s_j(1-q)})
=O(\rho^{2s_k(1-q)}).
\]
On the other hand, in view of \eqref{est:E} we have
$\|\Ep^1\|_{L^q(\Omega)}=o(1)$ and $\|\Ep^2\|_{L^q(\Omega)}=o(1)$ as $\rho\to0^+$.
Hence, the first estimate in  \eqref{est:q} is established.
The proof of the second estimate in \eqref{est:q} is similar.
\end{proof}
\section{The linear theory}
\label{sec:lineartheory}
We recall from \eqref{def:RSN} that the linear operator $\Lrho$ is defined 
for $\phi\in W^{2,p}(\Omega)$, $p>1$, by
\begin{equation}
\label{def:Lrho}
\Lrho\phi:=-\Delta\phi-\sum_{i=1}^k|x|^{\ai-2}e^{\wi}\phi
=-\Delta\phi-\sum_{i=1}^k\frac{2\ai^2\dei^{\ai}|x|^{\ai-2}}{(\dei^\ai+|x|^{\ai})^2}\phi,
\end{equation}
where $\ai,\dei$ are defined by \eqref{def:adodd}--\eqref{def:adeven} for $i=1,2,\ldots,k$.
\par
Our aim in this section is to establish the following result.
\begin{prop}
\label{prop:lineartheory}
Let $\Omega$ satisfy the symmetry assumption~\eqref{assumpt:Omega}.
For any $p>1$ there exist $\rho_0>0$ and $c>0$ such that for any $\rho\in(0,\rho_0)$
and for any $\psi\in L^p(\Omega)$ there exists a unique $\phi\in W^{2,p}(\Omega)\cap\calH$
solution to
\[
\Lrho\phi=\psi\ \mbox{in\ }\Omega,
\qquad\phi=0\ \mbox{on\ }\partial\Omega
\]
which satisfies
\[
\|\phi\|\le c|\ln\rho|\|\psi\|_p.
\]
\end{prop} 
We observe that $\Lrho$ is \textit{formally} the same operator appearing in \cite{GrossiPistoia}.
However, it actually depends significantly on the asymmetry parameter $\ga\in(0,1]$ via the parameters $\ai,\dei$.
Consequently, we can follow the approach in \cite{GrossiPistoia} to prove Proposition~\ref{prop:lineartheory},
although some intermediate estimates require a modified proof, due to the different dependence
of $\ai,\dei$ on $i,\rho$. In particular, since $\ai$ does not depend monotonically on $i$ (see Remark~\ref{rmk:ainotincreasing}),
the proof of Lemma~\ref{lem:GPexpansions}--(iv) below differs from the proof of the corresponding 
estimate~(4.18) in \cite{GrossiPistoia}.
\par
For the sale of completeness, in this section we first outline the scheme of the proof of 
Proposition~\ref{prop:lineartheory}, which is analogous to \cite{GrossiPistoia}.
We then devote the remaining part of this section to prove in detail Lemma~\ref{lem:GPexpansions}--(iv).
\subsection{Outline of the proof of Proposition~\ref{prop:lineartheory}}
It is convenient to extend the symmetry assumption~\eqref{assumpt:Omega}
to a possibly unbounded domain $D\subset\rr^2$.
Let $D\subset\rr^2$ be a smooth (possibly unbounded) domain.
Namely, we define the following geometrical symmetry property for $D$:
\begin{equation}
\label{def:Dsymm}
\begin{aligned}
0\in D\ \mbox{and\ }
&\begin{cases}
-D=D=e^{2\pi\sqrt{-1}/(m+n)}D,
&\mbox{if\ }\ga=\frac{m}{n},\ m,n\in\mathbb N\ \mbox{coprime};\\
-D=D,
&\mbox{if\ }\ga\not\in\mathbb Q,
\end{cases}
\end{aligned}
\end{equation}
\par
Correspondingly, we define a symmetry property for functions $\phi:D\to\mathbb R$: 
\begin{equation}
\label{def:phisymm}
\begin{cases}
\varphi(xe^{2\pi\sqrt{-1}/(m+n)})=\varphi(x)=\varphi(-x)\,\forall x\in D,
&\mbox{if\ }\ga=\frac{m}{n},\ m,n\in\mathbb N\ \mbox{coprime}\\
\varphi(-x)=\varphi(x)\,\forall x\in D,
&\mbox{if\ }\ga\not\in\mathbb Q.
\end{cases}
\end{equation}
The following lemma clarifies the role of the symmetry assumption \eqref{def:phisymm}.
\begin{lemma}
\label{lem:symm}
Suppose $\al\ge2$ is such that $\frac{\al}{2}\in\mathbb N$ and
\[
\frac{\al}{2}=(1+\frac{1}{\ga})i-\frac{1}{\ga}\qquad\qquad\mbox{for some odd $i\in\mathbb N$}
\]
or
\[
\frac{\al}{2}=(1+\ga)i-1\qquad\qquad\mbox{for some even $i\in\mathbb N$}.
\]
Suppose $\phi$ is a solution to 
\begin{equation}
\label{eq:phijo}
-\Delta\phi=2\al^2\frac{|y|^{\al-2}}{(1+|y|^\al)^2}\phi\qquad\mbox{in\ }\rr^2,
\qquad\qquad\int_{\rr^2}|\nabla\phi|^2<+\infty
\end{equation}
satisfying \eqref{def:phisymm} with $D=\rr^2$.
Then, there exists $\eta\in\rr$ such that
\begin{equation}
\label{eq:phirr}
\phi(y)=\eta\frac{1-|y|^\al}{1+|y|^\al}.
\end{equation}
\end{lemma}
\begin{proof}
It is shown in \cite{GrossiPistoia} that $\phi$ is necessarily a bounded solution.
In turn, it is shown in \cite{DelPinoEspositoMusso2012} that any bounded solution to
\eqref{eq:phirr} is a linear combination of the functions:
\[
\phi_0(y)=\frac{1-|y|^\al}{1+|y|^\al},
\qquad
\phi_1(y)=\frac{|y|^{\frac{\al}{2}}}{1+|y|^\al}\cos\frac{\al}{2}\theta,
\qquad
\phi_2(y)=\frac{|y|^{\frac{\al}{2}}}{1+|y|^\al}\sin\frac{\al}{2}\theta.
\]
In view of Corollary~\ref{cor:aiform},
$\frac{\al}{2}$ is of the form \eqref{eq:aiform}. In particular, the functions $\phi_1,\phi_2$
do not satisfy \eqref{def:phisymm}.
The claim follows.
\end{proof}
In order to prove Proposition~\ref{prop:lineartheory},
following \cite{GrossiPistoia}, we argue by contradiction and we assume that
there exist $p>1$, $\rho_n\to0^+$, $\phi_n$, $\psi_n$
such that $\mathcal L_{\rho_n}\phi_n=\psi_n$,
$\|\phi_n\|=1$,
$|\ln\rho_n|\|\psi_n\|_p\to0$.
In particular, $\phi_n$ satisfies
\begin{equation}
\label{eq:phin}
-\Delta\phi_n-\sum_{i=1}^k2\ai^2\frac{\de_{in}^\ai|x|^{\ai-2}}{(\de_{in}^{\ai}+|x|^{\ai})^2}\phi_n=\psi_n.
\end{equation}
Then, by \eqref{eq:deltadecay1} we obtain $k$ sequences of scaling parameters:
\[
\de_{jn}=d_j\rho_n^{s_j},\qquad j=1,2,\ldots,k.
\]
For every fixed $j=1,2,\ldots,k$ we set 
\[
\phi_n^j(y)=\phi_n(\de_{jn}y),\qquad y\in\Omega_n^j=\Omega/\de_{jn}.
\]
For any $\al\ge 2$ we define the Banach spaces
\[
L_\al(\rr^2):=\left\{u\in L_{\mathrm{loc}}^2(\rr^2):\left\|\frac{|y|^{\frac{\al-2}{2}}}{1+|y|^\al}u\right\|_{L^2(\rr^2)}<+\infty\right\}
\]
and
\[
H_\al(\rr^2):=\left\{u\in W_{\mathrm{loc}}^{1,2}(\rr^2):\|\nabla u\|_{L^2(\rr^2)}
+\left\|\frac{|y|^{\frac{\al-2}{2}}}{1+|y|^\al}u\right\|_{L^2(\rr^2)}<+\infty\right\}
\]
endowed with the norms
\[
\begin{aligned}
\|u\|_{L_\al}:=&\left\|\frac{|y|^{\frac{\al-2}{2}}}{1+|y|^\al}u\right\|_{L^2(\rr^2)}\\
\|u\|_{H_\al}:=&\left(\|\nabla u\|_{L^2(\rr^2)}
+\left\|\frac{|y|^{\frac{\al-2}{2}}}{1+|y|^\al}u\right\|_{L^2(\rr^2)}\right)^{1/2}.
\end{aligned}
\]
With these definitions, it is shown in \cite{GrossiPistoia} that the embedding $i_\al:H_\al(\rr^2)\hookrightarrow L_\al(\rr^2)$ is compact.
\begin{lemma}
\label{lem:etaj}
For every $j=1,2,\ldots,k$ there exist $\eta_j\in\rr$ such that
\begin{equation}
\label{eq:phirr}
\phij_n(y)\to\phi_0^j(y)=\eta_j\frac{1-|y|^{\aj}}{1+|y|^{\aj}}
\end{equation}
weakly in $H_{\aj}(\mathbb R^2)$ and strongly in $L_\aj(\mathbb R^2)$.
\end{lemma}
\begin{proof}
Fix $j$. There exists $\phij_0\in H_{\aj}$ such that
$\phij_n\to\phi_0$ weakly in $H_{\aj}(\mathbb R^2)$ and strongly in $L_\aj(\mathbb R^2)$.
The function $\phij_0$ satisfies \eqref{eq:phijo} with $\al=\aj$.
Moreover, for every $n$, $\Omega_n^j$ satisfies the symmetry assumption~\eqref{def:Dsymm}
and $\phij_n$ satisfies \eqref{def:phisymm} in $\Omega_n^j$.
Finally, $\aj$ is of the form \eqref{def:a}.
In view of Lemma~\ref{lem:symm} we conclude that $\phij_0$ is of the asserted form~\eqref{eq:phirr}
with $\al=\aj$.
\end{proof}
The desired contradiction will follow from the fact that, actually, $\phi_0^i=0$ for all $i=1,2,\ldots,k$.
Indeed, the following result holds true.
\begin{prop}
\label{prop:gajzero}
For any $j=1,2,\ldots,k$, there holds $\eta_j=0$.
Therefore, we have $\phij_n(y)\to0$ weakly in $H_{\aj}(\mathbb R^2)$ and strongly in $L_\aj(\mathbb R^2)$
for all $j=1,2,\ldots,k$.
\end{prop}
The proof of Proposition~\ref{prop:gajzero} will be outlined below. Once Proposition~\ref{prop:gajzero} is established,
it is simple to prove Proposition~\ref{prop:lineartheory}.
\begin{proof}[Proof of Proposition~\ref{prop:lineartheory}]
Multiplying \eqref{eq:phin} by $\phi_n$ and integrating, we find
\begin{equation*}
\begin{aligned}
1=&\int_\Omega|\nabla\phi_n|^2=\sum_{i=1}^k\int_\Omega2\ai^2
\frac{\de_{in}^{\ai}|x|^{\ai-2}}{(\de_{in}^{\ai}+|x|^{\ai})^2}\phi_n^2(x)\,dx
+\int_\Omega\psi_n(x)\phi_n(x)\,dx\\
=&\sum_{i=1}^k\int_{\Omega_n^i}2\ai^2
\frac{|y|^{\ai-2}}{(1+|y|^{\ai})^2}(\phi_n^i(y))^2(x)\,dy+O(\|\psi_n\|_p\|\phi_n\|)
=o(1),
\end{aligned}
\end{equation*}
because $\phi_n^i\to0$ strongly in $L_{\ai}(\rr^2)$, for all $i=1,2,\ldots,k$,
and $\|\psi_n\|_p=o(1)$.
This is a contradiction.
\end{proof}
In order to prove Proposition~\ref{prop:gajzero}, for $i=1,2,\ldots,k$ we define the quantities
\begin{equation*}
\scpi(\rho_n)=\ln\rho_n\int_{\Omj}2\ai^2\frac{|y|^{\ai-2}}{(1+|y|^{\ai})^2}\phi_n^i(y)\,dy.
\end{equation*}
\begin{lemma}
\label{lem:scplimit}
For any $i=1,2,\ldots,k$,
For any $i=1,2,\ldots,k$ there holds
\begin{equation*}
\scp_1(\rho_n)=o(1),
\qquad
\scpi(\rho_n)+2\sum_{j=1}^{i-1}\scpj(\rho_n)=o(1)
\end{equation*}
and consequently
\begin{equation}
\label{eq:scplimit}
\scp_{i0}:=\lim_{\rho_n\to0}\scpi(\rho_n)=0.
\end{equation}
\end{lemma}
We first show how Lemma~\ref{lem:scplimit} implies Proposition~\ref{prop:gajzero}.
Then, we devote the remaining part of this section to the proof of Lemma~\ref{lem:scplimit}.
\begin{proof}[Proof of Proposition~\ref{prop:gajzero}]
We use the following identities. For $i=1,2,\ldots,k$ there holds:
\begin{equation}
\label{eq:intid}
\begin{aligned}
&\int_{\rr^2}2\ai^2\frac{|y|^{\ai-2}}{(1+|y|^{\ai})^2}\frac{1-|y|^{\ai}}{1+|y|^{\ai}}\,dy=0\\
&\int_{\rr^2}2\ai^2\frac{|y|^{\ai-2}}{(1+|y|^{\ai})^2}\frac{1-|y|^{\ai}}{1+|y|^{\ai}}\ln(1+|y|^{\ai})^2\,dy=-4\pi\ai\\
&\int_{\rr^2}2\ai^2\frac{|y|^{\ai-2}}{(1+|y|^{\ai})^2}\frac{1-|y|^{\ai}}{1+|y|^{\ai}}\ln|y|\,dy=-4\pi.
\end{aligned}
\end{equation}
Using the equations for $\phi_n$ and $Pw_{in}$, we find
\[
\int_\Omega2\ai^2\frac{\de_{in}^{\ai}|x|^{\ai-2}}{(\de_{in}^{\ai}+|x|^{\ai})^2}\phi_n
=\sum_{j=1}^k\int_\Omega2\aj^2\frac{\de_{jn}^{\aj}|x|^{\aj-2}}{(\dej^{\aj}+|x|^{\aj})^2}\phi_nPw_{in}+\int_\Omega\psi Pw_{in},
\]
where $w_{in}=w_{\de_{in}}^{\ai}$.
The first term above vanishes as $\rho_n\to0^+$,
in view of the form~\eqref{eq:phirr} of $\phi_0^i$ and of the first integral in \eqref{eq:intid}.
In order to evaluate the second term, we note that similarly as in \cite{GrossiPistoia} we find 
\[
\begin{aligned}
\int_\Omega&2\aj^2\frac{\de_{jn}^{\aj}|x|^{\aj-2}}{(\de_{jn}^{\aj}+|x|^{\aj})^2}\phi_n Pw_{in}\\
&=\begin{cases}
-2(2(k-i)+1)\sigma_j(\rho_n)+o(1)
&\mbox{if $j<i$}\\
-2(2(k-i)+1)\sigma_i(\rho_n)
+\int_{\Omega_n^i}2\ai^2\frac{\de_{in}^{\ai}|x|^{\ai-2}}{(\de_{in}^{\ai}+|x|^{\ai})^2}\phi_n^i(y)
[-2\ln(1+|y|^{\ai})]\,dy+o(1),
&\mbox{if $j=i$}\\
-2(2(k-j)+1)\sigma_j(\rho_n)
+\int_{\Omega_n^j}2\aj^2\frac{\de_{jn}^{\aj}|x|^{\aj-2}}{(\de_{jn}^{\aj}+|x|^{\aj})^2}\phi_n^j(y)
[-2\ai\ln|y|]\,dy+o(1),
&\mbox{if $j>i$}
\end{cases}
\end{aligned}
\]
Therefore, Lemma~\ref{lem:scplimit} and \eqref{eq:intid} yield
\[
\begin{aligned}
\sum_{j=1}^k\int_\Omega2\aj^2\frac{\de_{jn}^{\aj}|x|^{\aj-2}}{(\de_{jn}^{\aj}+|x|^{\aj})^2}\phi_nPw_{in}
=\begin{cases}
4\pi\ai(\eta_i+2\sum_{j=i+1}^k\eta_j)+o(1),
&\mbox{if $i=1,2,\ldots,k-1$}\\
4\pi\al_k\eta_k+o(1),
&\mbox{if $i=k$}.
\end{cases}
\end{aligned}
\]
In turn, we obtain
\[
\eta_k=0
\qquad\qquad\mbox{and\ }\qquad\qquad
\eta_i+2\sum_{j=i+1}^k\eta_j=0
\qquad\qquad\mbox{for any\ }i=1,2,\ldots,k-1.
\]
Proposition~\ref{prop:gajzero} is thus established.
\end{proof}
We are left to prove the asymptotic behavior of the quantities $\scp_{i}(\rho_n)$,
$i=1,2,\ldots,k$, as stated in Lemma~\ref{lem:scplimit}.
\subsection{Proof of Lemma~\ref{lem:scplimit}}
Throughout this subsection, for the sake of simplicity, we omit the index $n$.
In order to establish Lemma~\ref{lem:scplimit} 
we set
\begin{equation}
\label{def:Zi}
\Zi(x)=\frac{\dei^{\ai}-|x|^{\ai}}{\dei^{\ai}+|x|^{\ai}}
\end{equation}
and we denote by $P\Zi$ its projection onto $H_0^1(\Omega)$.
Then, using the equations for $\phi$ and $P\Zi$, we find that the sequence $\phi$ satisfies the identity
\begin{equation*}
\begin{aligned}
\ln\rho\int_\Omega2\ai^2\frac{\dei^{\ai}|x|^{\ai-2}}{(\dei^{\ai}+|x|^{\ai})^2}\phi(x)\Zi(x)\,dx
=&\ln\rho\sum_{j=1}^k\int_\Omega2\aj^2\frac{\dej^{\aj}|x|^{\aj-2}}{(\dej^{\aj}+|x|^{\aj})^2}\phi(x)P\Zi(x)\,dx\\
&+\ln\rho\int_\Omega\psi(x)P\Zi(x)\,dx.
\end{aligned}
\end{equation*}
Equivalently, we may write
\begin{equation*}
\begin{aligned}
\ln\rho\int_\Omega2\ai^2&\frac{\dei^{\ai}|x|^{\ai-2}}{(\dei^{\ai}+|x|^{\ai})^2}\phi(x)(P\Zi(x)-\Zi(x))\,dx\\
+&\ln\rho\sum_{j\neq i}^k\int_\Omega2\aj^2\frac{\dej^{\aj}|x|^{\aj-2}}{(\dej^{\aj}+|x|^{\aj})^2}\phi(x)P\Zi(x)\,dx
+\ln\rho\int_\Omega\psi(x)P\Zi(x)\,dx=0.
\end{aligned}
\end{equation*}
The asserted identities~\eqref{eq:scplimit} will then follow from the following facts.
\begin{lemma}
\label{lem:GPexpansions}
The following expansions hold.
\begin{enumerate}
  \item[(i)]
$\ln\rho\int_\Omega2\ai^2\frac{\dei^{\ai}|x|^{\ai-2}}{(\dei^{\ai}+|x|^{\ai})}\phi(x)(P\Zi(x)-\Zi(x))\,dx
=\scpi(\rho)+o(1)$ 
\item[(ii)]
$\ln\rho\int_\Omega\psi(x)P\Zi(x)\,dx=o(1)$
\item[(iii)] If $j>i$, then
\[ 
\ln\rho\int_\Omega2\aj^2\frac{\dej^{\aj}|x|^{\aj-2}}{(\dej^{\aj}+|x|^{\aj})^2}\phi(x)P\Zi(x)\,dx=o(1)
\]
\item[(iv)] If $j<i$, then
\begin{equation}
\label{eq:desiredest}
\ln\rho\int_\Omega2\aj^2\frac{\dej^{\aj}|x|^{\aj-2}}{(\dej^{\aj}+|x|^{\aj})^2}\phi(x)P\Zi(x)\,dx
=2\scpj(\rho)+o(1)
\end{equation}
\end{enumerate}
\end{lemma}
The proof of Lemma~\ref{lem:GPexpansions}--(i)--(ii)--(iii) is completely analogous to \cite{GrossiPistoia}.
On the other hand, the proof of Lemma~\ref{lem:GPexpansions}--(iv) is different,
due to the particular dependence on $\ga$ of $\ai,\dei$.
Therefore, we provide the proof of Lemma~\ref{lem:GPexpansions}--(iv).
The underlying idea is that in order to control the integrals on the expanding domain
$\Omega_n^j$ it is convenient to decompose $\Omega_n^j=B_{R_j}\cup(\Omega_n^j\setminus B_{R_j})$,
with $R_j$ suitably defined as follows.
\par
For any $j=1,2,\ldots,k-1$ we define
\begin{equation}
\label{def:Rj}
\Rj:=\sqrt{\frac{\de_{j+1}}{\dej}}.
\end{equation}
Then, $\Rj\to+\infty$ as $\rho\to0$.
\par
We recall some elementary facts.
\begin{lemma}
\label{lem:elemfacts}
The following properties hold.
\begin{enumerate}
  \item[(i)] For $R\to+\infty$ and for any $\beta>0$ there holds
\[
\int_{\mathbb R^2\setminus B_R}\frac{dy}{|y|^{2+\beta}}=O(\frac{1}{R^\beta});
\]  
  \item[(ii)] If $j<i$, then for all $y\in B_{\Rj}$
\[
|y|^{\ai}\left(\frac{\dej}{\dei}\right)^{\ai}\le\left(\frac{\dej}{\dei}\right)^{\ai/2}=O(\frac{1}{\Rj^{\ai}});
\] 
\item[(iii)] For any $q>1$ there holds
\[
\|\phij\|_{L^q(\Omj)}=\frac{1}{\dej^{2/q}}\|\phi\|_{L^q(\Omega)}.
\] 
 \end{enumerate}
\end{lemma}
\begin{proof}
Part~(i) is elementary.
Proof of (ii). See \eqref{est:igej} or \cite{GrossiPistoia}.
Proof of (iii). We use H\"older's inequality.
\end{proof}
\begin{lemma}
\label{lem:Ziexp}
The following expansions hold for the function $\Zi(x)$ defined in \eqref{def:Zi}.
\begin{enumerate}
\item [(i)] For any $x\in\Omega$ there holds
\[
P\Zi(x)=\Zi(x)+1+O(\dei^{\ai})=\frac{2\dei^{\ai}}{\dei^{\ai}+|x|^{\ai}}+O(\dei^{\ai});
\] 
In particular, $|P\Zi(x)|\le2+O(\dei^{\ai})$, $i=1,2,\ldots,k$.
\item[(ii)] If $j<i$ and $y\in B_{\Rj}$, then
\[
P\Zi(\dej y)=\frac{2\dei^{\ai}}{\dei^{\ai}+\dej^{\ai}|y|^{\ai}}+O(\dei^{\ai})
=\frac{2}{1+\left(\frac{\dej}{\dei}\right)^{\ai}|y|^{\ai}}+O(\dei^{\ai})
=\frac{2}{1+O(\Rj^{-\ai})}
\]
and
\[
P\Zi(\dej y)-2=-\frac{2\left(\frac{\dej}{\dei}\right)^{\ai}|y|^{\ai}}{1+\left(\frac{\dej}{\dei}\right)^{\ai}|y|^{\ai}}+O(\dei^{\ai})
=O(\frac{1}{\Rj^{\ai}})+O(\dei^{\ai}),
\]
where $\Rj$ is defined in \eqref{def:Rj}.
\end{enumerate}
\end{lemma}
\begin{proof}
The proof readily follows from the definition of $\Zi$ in \eqref{def:Zi} and Lemma~\ref{lem:elemfacts}--(ii).
\end{proof}
Rescaling the integral on the l.h.s.\ in \eqref{eq:desiredest}, we have
\[
\int_\Omega\frac{\dej^{\aj}|x|^{\aj-2}}{(\dej^{\aj}+|x|^{\aj})^2}\phi(x)P\Zi(x)\,dx
\stackrel{x=\dej y}{=}\int_{\Omj}\frac{|y|^{\aj-2}}{(1+|y|^{\aj})^2}\phij(y)P\Zi(\dej y)\,dy.
\]
The proof of \eqref{eq:desiredest} will finally follow from the following lemmas.
\begin{lemma}
\label{lem:PhiZifirst}
There exists $\beta_1>0$ such that
\begin{align}
\label{est:phiZfirst}
&\int_{\Omj\setminus B_{Rj}}\frac{|y|^{\aj-2}}{(1+|y|^{\aj})^2}\phij(y)P\Zi(\dej y)\,dy=O(\rho^{\beta_1});\\
\label{est:phiZsecond}
&\int_{B_{Rj}}\frac{|y|^{\aj-2}}{(1+|y|^{\aj})^2}\phij(y)(P\Zi(\dej y)-2)\,dy=O(\rho^{\beta_1}).
\end{align}
\end{lemma}
\begin{proof}
We have, using Lemma~\ref{lem:Ziexp}--(i):
\begin{align*}
\int_{\Omj\setminus B_{Rj}}\left|\frac{|y|^{\aj-2}}{(1+|y|^{\aj})^2}\phij(y)P\Zi(\dej y)\right|\,dy
\le&\int_{\Omj\setminus B_{Rj}}\frac{|y|^{\aj-2}}{(1+|y|^{\aj})^2}|\phij(y)|(2+O(\dei^{\ai}))\,dy\\
&=O\left(\int_{\Omj\setminus B_{Rj}}\frac{|y|^{\aj-2}}{(1+|y|^{\aj})^2}|\phij(y)|\right)\,dy.
\end{align*}
In view of Lemma~\ref{lem:elemfacts} and H\"older's inequality we derive that
for any $r>1$ there holds
\begin{align*}
\int_{\Omj\setminus B_{Rj}}&\frac{|y|^{\aj-2}}{(1+|y|^{\aj})^2}|\phij(y)|\,dy
\le\|\phij\|_{L^{r/(r-1)}(\Omj)}
\left(\int_{\Omj\setminus B_{Rj}}\Big[\frac{|y|^{\aj-2}}{(1+|y|^{\aj})^2}\Big]^r\right)^{1/r}\\
=&O\left(\frac{1}{\dej^{2(r-1)/r}}\Big(\int_{\Omj\setminus B_{Rj}}\frac{dy}{|y|^{(\aj+2)r}}\Big)\right)^{1/r}
=O\left(\frac{1}{\dej^{2(r-1)/r}\Rj^{\aj+2-2/r}}\right).
\end{align*}
By taking $0<r-1\ll1$, we obtain estimate~\eqref{est:phiZfirst} for some $\beta_1>0$.
\par
Similarly we have, using Lemma~\ref{lem:Ziexp}--(ii):
\begin{align*}
\int_{B_{Rj}}\frac{|y|^{\aj-2}}{(1+|y|^{\aj})^2}&\phij(y)(P\Zi(\dej y)-2)\,dy\\
&=O\left(\int_{B_{Rj}}\frac{|y|^{\aj-2}}{(1+|y|^{\aj})^2}\phij(y)[O(\frac{1}{\Rj^{\ai}})+O(\dei^{\ai})]\,dy\right)\\
&=\left(O(\frac{1}{\Rj^{\ai}})+O(\dei^{\ai})\right)\int_{B_{Rj}}\frac{|y|^{\aj-2}}{(1+|y|^{\aj})^2}|\phij(y)|\,dy\\
&\le\|\phij\|_{L^{r/(r-1)}(\Omj)}(\int_{\mathbb R^2}[\frac{|y|^{\aj-2}}{(1+|y|^{\aj})^2}]^r\,dy)^{1/r}\left(O(\frac{1}{\Rj^{\ai}})+O(\dei^{\ai})\right)\\
&=\frac{1}{\dej^{2(r-1)/r}}\Big(O(\frac{1}{\Rj^{\ai}})+O(\dei^{\ai})\Big).
\end{align*}
By taking $0<r-1\ll1$ and possibly a smaller value for $\beta_1>0$, we deduce estimate~\eqref{est:phiZsecond}.
\end{proof}
We conclude from Lemma~\ref{lem:PhiZifirst}
that if $j<i$, then
\begin{equation}
\ln\rho\int_\Omega\frac{2\aj^2\dej^{\aj}|x|^{\aj-2}}{(\dej^{\aj}+|x|^{\aj})^2}\phi(x)P\Zi(x)\,dx
=2\ln\rho\int_{B_{\Rj}}\frac{2\aj^2|y|^{\aj-2}}{(1+|y|^{\aj})^2}\phij(y)\,dy+o(1).
\end{equation}
\begin{lemma}
\label{lem:phifinal}
There exists $\beta_2>0$ such that
\[
\int_{\Omj}\frac{|y|^{\aj-2}}{(1+|y|^{\aj})^2}\phij(y)\,dy
=\int_{B_{\Rj}}\frac{|y|^{\aj-2}}{(1+|y|^{\aj})^2}\phij(y)\,dy+O(\rho^{\beta_2}).
\]
Consequently,
\[
\ln\rho\int_\Omega\frac{2\aj^2\dej^{\aj}|x|^{\aj-2}}{(\dej^{\aj}+|x|^{\aj})^2}\phi(x)P\Zi(x)\,dx
=2\scpj(\rho)+o(1).
\]
\end{lemma}
\begin{proof}
We have, for $r>1$ sufficiently small,
\begin{align*}
|\int_{\Omj\setminus B_{\Rj}}\frac{|y|^{\aj-2}}{(1+|y|^{\aj})^2}\phij(y)\,dy|
\le&\|\phij\|_{L^{r/(r-1)}(\Omj)}\left(\int_{\mathbb R^2\setminus B_{\Rj}}\frac{dy}{|y|^{(\aj+2)r}}\right)^{1/r}\\
=&O\left(\frac{1}{\dej^{2(r-1)/r}\Rj^{\aj+1-1/r}}\right).
\end{align*}
The statement follows by taking $0<r-1\ll1$.
\end{proof}
Estimate~\eqref{eq:desiredest} in Lemma~\ref{lem:GPexpansions}--(iv) is thus completely established.
In turn, the proof of \eqref{eq:scplimit} follows.
Hence, the proof of Proposition~\ref{prop:lineartheory} is complete.
\section{The contraction mapping and the proof of Theorem~\ref{thm:main}}
\label{sec:fixpoint}
In this section we conclude the proof of Theorem~\ref{thm:main} by obtaining a solution $\ur$
to problem~\eqref{eq:pb}
in the form $\ur=\Wr+\phir$, with $\phir$ the fixed point of a contraction mapping.
Indeed, we establish the following existence result.
\begin{prop}
\label{prop:fixedpoint}
For $p>1$ sufficiently close to 1 there exist $\rho_0>0$ and $R>0$ such that
for any $\rho\in(0,\rho_0)$ there exists a unique solution $\phir\in\calH$
to the problem
\[
-\Delta(\Wr+\phir)=\rho f(\Wr+\phir)\ \mathrm{in\ }\Omega,\qquad\phir=0\ \mathrm{on\ }\partial\Omega
\]
satisfying the estimate
\[
\|\phir\|\le R\,\rho^{\overline\beta_p}|\ln\rho|.
\]
\end{prop}
Here $\calH$ is the space of $\ga$-symmetric Sobolev functions defined in \eqref{def:H}
and $\bar\beta_p>0$ is the exponent obtained in Proposition~\ref{prop:RS}.
\par
We equivalently seek a fixed point $\phi\in\calH$ for the operator 
$\Trho:\calH\to\calH$ defined by
\[
\Trho(\phi)=(\Lrho)^{-1}(\Nrho(\phi)+\Srho\phi+\Rrho),
\]
where $\Rrho,\Srho,\Nrho$ are the operators defined in \eqref{def:R}--\eqref{def:RSN}.
\par
In the sequel we shall use the Moser-Trudinger inequality \cite{Moser,Trudinger}
in the following form.
\begin{lemma}
\label{lem:MT}
There exists $c>0$ such that for any bounded domain $\Omega\subset\mathbb R^2$
there holds
\[
\int_\Omega e^{4\pi u^2/\|u\|^2}\,dx\le c|\Omega|,\quad\forall u\in H_0^1(\Omega).
\]
In particular, there exists $c>0$ such that for any $\eta\in\mathbb R$
\[
\int_\Omega e^{\eta u}\,dx\le c|\Omega|e^{\frac{\eta^2}{16\pi}\|u\|^2},
\quad\forall u\in H_0^1(\Omega).
\]
\end{lemma}
\begin{lemma}
\label{lem:Nest}
For any $p\ge1$ and $r>1$ there exists $\rho_0>0$ and $c_1,c_2>0$ such that
for any $\rho\in(0,\rho_0)$ we have
\begin{equation}
\label{est:N1}
\|\Nrho(\phi)\|_p\le c_1e^{c_2\|\phi\|^2}\rho^{2s_k\frac{1-pr}{pr}}\|\phi\|^2
\end{equation}
for all $\phi\in H_0^1(\Omega)$ and
\begin{equation}
\label{est:N2}
\|\Nrho(\phi_1)-\Nrho(\phi_2)\|_p\le c_1e^{c_2(\|\phi_1\|^2+\|\phi_1\|^2)}\rho^{2s_k\frac{1-pr}{pr}}
\|\phi_1-\phi_2\|(\|\phi_1\|+\|\phi_2\|),
\end{equation}
for all $\phi_1,\phi_2\in H_0^1(\Omega)$.
\end{lemma}
Here, $s_k>0$ is the constant defined in \eqref{def:sj}.
\begin{proof}[Proof of Lemma~\ref{lem:Nest}]
Since \eqref{est:N1} follows from \eqref{est:N2} by taking $\phi_2=0$, 
it suffices to prove \eqref{est:N2}.
\par
We readily check that
\[
f(t+s)-f(t)-f'(t)s
=e^t(e^s-1-s)-\tau e^{-\ga t}(e^{-\ga s}-1+\ga s).
\]
Consequently,
\begin{equation*}
\Nrho(\phi_1)-\Nrho(\phi_2)
=\rho e^{\Wr}[e^{\phi_1}-e^{\phi_2}-(\phi_1-\phi_2)]
-\tau e^{-\ga\Wr}[e^{-\ga\phi_1}-e^{-\ga\phi_2}+\ga(\phi_1-\phi_2)].
\end{equation*}
Using the Mean Value Theorem, we have
\begin{equation}
\label{eq:MVT}
|e^a-e^b-a+b|\le e^{|a|+|b|}|a-b|(|a|+|b|),
\end{equation}
for all $a,b\in\mathbb R$.
Taking $a=\phi_1$, $b=\phi_2$ we derive
\[
|e^{\phi_1}-e^{\phi_2}-(\phi_1-\phi_2)|\le e^{|\phi_1|+|\phi_2|}|\phi_1-\phi_2|(|\phi_1|+|\phi_2|).
\]
Setting
\[
I_1:=\rho e^{\Wr}(e^{\phi_1}-e^{\phi_2}-(\phi_1-\phi_2)),
\]
we estimate:
\begin{equation*}
\begin{aligned}
\|I_1\|_p=&\left(\int_\Omega\rho^pe^{p\Wr}|e^{\phi_1}-e^{\phi_2}-\phi_1+\phi_2|^p\,dx\right)^{1/p}\\
\le&\sum_{j=1}^2\left(\int_\Omega\rho^pe^{p\Wr}e^{p|\phi_1|+p|\phi_2|}|\phi_1-\phi_2|^p|\phi_j|^p\,dx\right)^{1/p}.
\end{aligned}
\end{equation*}
By H\"older's inequality with $r^{-1}+s^{-1}+t^{-1}=1$ we obtain
\[
\|I_1\|_p\le C\sum_{j=1}^2\left(\int_\Omega\rho^{pr}e^{pr\Wr}\right)^{\frac{1}{pr}}
\left(\int_\Omega e^{ps(|\phi_1|+|\phi_2|)}\,dx\right)^{\frac{1}{ps}}
\left(\int_\Omega|\phi_1-\phi_2|^{pt}|\phi_j|^{pt}\,dx\right)^{\frac{1}{pt}}.
\]
In view of \eqref{est:q} we have
\[
\left(\int_\Omega(\rho^{pr}e^{pr\Wr})\,dx\right)^{1/(pr)}=O(\rho^{2s_k\frac{1-pr}{pr}}),
\]
where $s_k>0$ is defined in \eqref{def:sj}.
Now, the Moser-Trudinger inequality as in Lemma~\ref{lem:MT} yields
\[
\begin{aligned}
&\int_\Omega e^{ps(|\phi_1|+|\phi_2|)}\,dx\le Ce^{(ps)^2/(8\pi)(\|\phi_1\|^2+\|\phi_2\|^2)}\\
&\left(\int_\Omega|\phi_1-\phi_2|^{pt}|\phi_j|^{pt}\,dx\right)^{\frac{1}{pt}}
\le C\|\phi_1-\phi_2\|\,\|\phi_j\|.
\end{aligned}
\]
We conclude that
\[
\|I_1\|_p\le C\rho^{2s_k\frac{1-pr}{pr}}
e^{(ps)^2/(8\pi)(\|\phi_1\|^2+\|\phi_2\|^2)}\|\phi_1-\phi_2\|\,(\|\phi_1\|+\|\phi_2\|).
\]
Let $I_2$ be defined by
\[
I_2:=e^{-\ga\Wr}[e^{-\ga\phi_1}-e^{-\ga\phi_2}+\ga(\phi_1-\phi_2)].
\]
Taking $a=-\ga\phi_1$, $b=-\ga\phi_2$ in \eqref{eq:MVT} we derive
\[
|e^{-\ga\phi_1}-e^{-\ga\phi_2}+\ga(\phi_1-\phi_2)|
\le\ga^2e^{\ga(|\phi_1|+|\phi_2|)}|\phi_1-\phi_2|(|\phi_1|+|\phi_2|).
\]
Hence, by analogous estimates as above, we conclude the proof of the desired estimates.
\end{proof}
Now we can prove the main result of this section.
\begin{proof}[Proof of Proposition~\ref{prop:fixedpoint}]
Let 
\[
B_{\rho,R}:=\left\{\phi\in\calH:\|\phi\|\le R\rho^{\overline\beta_p}|\ln\rho|\right\}.
\]
We shall prove that $\Trho$ is a contraction mapping in $B_{\rho,R}$,
provided $R>0$ is sufficiently large and $\rho>0$ is sufficiently small.
\par
\textit{Claim~1.}
$\Trho$ maps $B_{\rho,R}$ into itself.
\par
Equivalently, we claim that
\[
\|\phi\|\le R\rho^{\overline\beta_p}|\ln\rho|
\Longrightarrow\|\Trho(\phi)\|\le R\rho^{\overline\beta_p}|\ln\rho|.
\]
Indeed, we have
\[
\begin{aligned}
\|\Trho(\phi)\|\le&\|\Lrho^{-1}\|(\|\Nrho(\phi)\|_p+\|\Srho\phi\|_p+\|\Rrho\|_p)\\
\le&C|\ln\rho|\left(\|\phi\|^2e^{c_2\|\phi\|^2}\rho^{2s_k\frac{1-pr}{pr}}+\|\Srho\|_{pq}\|\phi\|_{pq'}+\rho^{\overline\beta_p}\right)\\
\le&C|\ln\rho|\left(\|\phi\|^2e^{c_2\|\phi\|^2}\rho^{2s_k\frac{1-pr}{pr}}+\rho^{\overline\beta_{pq}}\|\phi\|+\rho^{\overline\beta_p}\right)\\
\le&R\rho^{\overline\beta_p}|\ln\rho|.
\end{aligned}
\]
\par
\textit{Claim~2.}
$\Trho$ is a contraction in $B_{\rho,R}$.
\par
Equivalently, we claim that there exists $L<1$ such that
\[
\|\phi_1\|, \|\phi_2\|\le R\rho^{\overline\beta_p}|\ln\rho|
\Longrightarrow
\|\Trho(\phi_1)-\Trho(\phi_2)\|\le L\|\phi_1-\phi_2\|.
\]
Indeed, we have
\[
\begin{aligned}
\|\Trho(\phi_1)-\Trho(\phi_2)\|\le&\|\Lrho^{-1}\|(\|\Nrho(\phi_1)-\Nrho(\phi_2)\|_p+\|\Srho(\phi_1-\phi_2)\|_p)\\
\le&C|\ln\rho|(\|\Nrho(\phi_1)-\Nrho(\phi_2)\|_p+\|\Srho(\phi_1-\phi_2)\|_p).
\end{aligned}
\]
We estimate:
\[
\begin{aligned}
\|\Nrho(\phi_1)-\Nrho(\phi_2)\|_p\le&c_1e^{c_2(\|\phi_1\|^2+\|\phi_2\|^2)}\rho^{2s_k\frac{1-pr}{pr}}
\|\phi_1-\phi_2\|(\|\phi_1\|+\|\phi_2\|)\\
\le&c_1e^{2c_2R^2}2R\rho^{\overline\beta_p}|\ln\rho|\rho^{2s_k\frac{1-pr}{pr}}\|\phi_1-\phi_2\|\\
\le&C\rho^{\overline\beta_p/2}\|\phi_1-\phi_2\|
\end{aligned}
\]
provided $p,r>1$ are sufficiently close to 1 and $\rho>0$ is sufficiently small.
\par
On the other hand, we have
\[
\begin{aligned}
\|\Srho(\phi_1-\phi_2)\|_p\le\|\Srho\|_{pq}\|\phi_1-\phi_2\|_{pq'}\le C\rho^{\overline\beta_{pq}}\|\phi_1-\phi_2\|,
\end{aligned}
\]
where $\frac{1}{q}+\frac{1}{q'}=1$.
We conclude that $\Trho$ is a contraction in $B_{\rho,R}$.
\par
Now the existence of $\phir$ follows by the Banach Contraction Principle.
\end{proof}
\begin{lemma}
\label{lem:totalmass}
There holds
\[
\Wr(x)=\calMk G(x,0)+o(1),
\]
uniformly on compact subsets of $\Omega\setminus\{0\}$,
where $\calMk$ is defined in \eqref{def:Mk}.
\end{lemma}
\begin{proof}
In view of \eqref{eq:proj}, we have
\begin{equation*}
P w_\de^\al(x)-4\pi\al G(x,0)=2\ln\frac{|x|^\al}{\de^\al+|x|^\al}+O(\de^\al)=O(\de^\al),
\end{equation*} 
uniformly on compact subsets of $\Omega\setminus\{0\}$.   
It follows that
\begin{equation*}
\Wr(x)=-4\pi\sum_{i=1}^k\frac{(-1)^i}{\ga^{\si}}\ai G(x,0)+o(1)
=\calMk G(x,0)+o(1),
\end{equation*} 
as asserted.                     
\end{proof}
Finally, we are able to provide the proof of our main result.
\begin{proof}[Proof of Theorem~\ref{thm:main}]
For $p>1$ sufficiently close to 1 let $\rho_0>0$ be chosen as in Proposition~\ref{prop:fixedpoint}.
For all $\rho\in(0,\rho_0)$ we obtain by Proposition~\ref{prop:fixedpoint} a solution $\ur=\Wr+\phir\in\calH$
to problem~\eqref{eq:pb} satisfying $\|\phir\|\le R\rho^{\overline\beta}|\ln\rho|$
with $\overline\beta=\overline\beta_p$.
In view of Ansatz~\eqref{def:Wr}, and taking into account Lemma~\ref{lem:totalmass}, we conclude that $\ur$
has the desired concentration properties.
\end{proof}
\section{Appendix}
We collect in this Appendix the proof of some complementary results stated in Section~\ref{sec:intro},
as well as some proofs.
\subsection{Remarks on the blow-up masses}
We denote by $\omega(x)=H(x,x)$ the Robin's function, where $H(x,y)$ is the regular part of the Green's function
as defined in \eqref{def:regG}.
We first provide a simple proof of identity~\eqref{eq:massid}.
We note that the proof of identity~\eqref{eq:massid} may be also deduced as a special case of an identity
involving probability measures established in \cite{ORS}.
\begin{lemma}
\label{lem:blowup}
Let $\ur$ be a solution sequence for \eqref{eq:pb}.
Suppose 
\[
\rho e^{\ur}\stackrel{\ast}{\rightharpoonup}m_+(x_0)\de_{x_0},
\qquad\qquad
\rho\tau e^{-\ga\ur}\stackrel{\ast}{\rightharpoonup}m_-(x_0)\de_{x_0}
\] 
for some $x_0\in\Omega$, weakly in the sense of measures, as $\rho\to0^+$.
Then,
\[
8\pi(m_+(x_0)+\frac{m_-(x_0)}{\ga})=(m_+(x_0)-m_-(x_0))^2
\]
and $x_0$ is a critical point for $\omega(x)$.
\end{lemma}
\begin{proof}
We adapt the argument in \cite{Ye}.
Without loss of generality, we may assume $x_0=0$.
We recall that $\calMk=m_+(0)-m_-(0)$.
Recall that 
$f(t)=e^t-\tau e^{-\ga t}$, $F(t)=e^t+\tau\ga^{-1}e^{-\ga t}$,
so that $F'(t)=f(t)$.
Then, as $\rho\to0^+$,
\[
\rho f(\ur)\stackrel{\ast}{\rightharpoonup}\calMk\de_{0},
\qquad\qquad
\rho F(\ur)\stackrel{\ast}{\rightharpoonup}(m_+(0)+\frac{m_-(0)}{\ga})\de_{0},
\]
weakly in the sense of measures.
We use the standard complex notation $z=x_1+\sqrt{-1}x_2$, 
$\partial_z=(\partial_{x_1}-\sqrt{-1}\partial_{x_2})/2$, $\partial_{\bar z}=(\partial_{x_1}+\sqrt{-1}\partial_{x_2})/2$,
$\partial_z\partial_{\bar z}=\Delta/4$.
We set 
\[
\mathcal H:=\frac{u_z^2}{2},
\qquad\qquad 
\mathcal K:=N_z\ast[\rho F(\ur)\chi_\Omega]_z,
\]
where $N(z,\bar z)=(4\pi)^{-1}\ln(z\bar z)$ is the Newtonian potential satisfying $\Delta N=\de_0$.
Then, $\mathcal S_{\bar z}=0$, i.e., $\mathcal S:=\mathcal H+\mathcal K$ is holomorphic in $\Omega$.
Hence, $\mathcal S$ converges uniformly to a holomorphic function $\mathcal S_0$ in $\Omega$
as $\rho\to0^+$.
Since $\ur\to\calMk G(x,0)$ in $W^{1,q}(\Omega)$ and uniformly in $\Omega\setminus\{0\}$,
we have $\mathcal H\to\mathcal H_0$,
where
\[
\mathcal H_0=\frac{u_{0,z}^2}{2}=\frac{\calMk^2}{2}G_z^2(x,0)=\frac{\calMk^2}{2}(N_z+H_z(x,0))^2.
\]
Since $N_z=(4\pi z)^{-1}$, we conclude that
\[
\mathcal H_0=\frac{\calMk^2}{32\pi^2z^2}+\frac{\calMk^2}{4\pi z}H_z(x,0)+\frac{\calMk^2}{2}H_z^2(x,0).
\]
On the other hand, we have
\[
\mathcal K:=N_{zz}\ast[\rho F(\ur)\chi_\Omega]=-\frac{1}{4\pi z^2}\ast[\rho F(\ur)\chi_\Omega].
\]
Taking limits, we find $\mathcal K\to\mathcal K_0$, where
\[
\mathcal K_0=-\frac{1}{4\pi z^2}\ast[(m_+(0)+\frac{m_-(0)}{\ga})\de_{0}]=-\frac{m_+(0)+\frac{m_-(0)}{\ga}}{4\pi z^2}.
\]
We conclude that
\[
\mathcal S_0=\frac{1}{4\pi z^2}\left[\frac{\calMk^2}{8\pi}-(m_+(0)+\frac{m_-(0)}{\ga})\right]
+\frac{\calMk^2}{4\pi z}H_z(x,0)+\frac{\calMk^2}{2}H_z^2(x,0).
\]
Since $\mathcal S_0$ is smooth in $\Omega$, we necessarily have
\[
\frac{\calMk^2}{8\pi}-(m_+(0)+\frac{m_-(0)}{\ga})=0
\qquad\qquad\mbox{and\ }H_z(x,0)|_{x_0}=0.
\]
The asserted identities follow.
\end{proof}
The following is a proof of \eqref{def:qkgi} in Proposition~\ref{prop:deimain}.
\begin{proof}[Proof of \eqref{def:qkgi}]
Suppose $k$ is odd.
We compute:
\begin{equation*}
\begin{aligned}
2q_j=&2(s_j-s_{j+1})
=\frac{(1+\ga)(k-j)+\ga}{(1+\ga)j-1}-\frac{(1+\ga)(k-j-1)+\ga}{(1+\ga)(j+1)-1}\\
=&\frac{[(1+\ga)(k-j)+\ga][(1+\ga)(j+1)-1]-[(1+\ga)(k-j-1)+\ga][(1+\ga)j-1]}{[(1+\ga)j-1][(1+\ga)(j+1)-1]}
\end{aligned}
\end{equation*}
Observing that $(1+\ga)(j+1)-1=(1+\ga)(j+1)+\ga$
and $(1+\ga)(k-j-1)+\ga=(1+\ga)(k-j)-1$
we find
\begin{align*}
[(1+\ga)&(k-j)+\ga][(1+\ga)(j+1)-1]-[(1+\ga)(k-j-1)+\ga][(1+\ga)j-1]\\
=&[(1+\ga)(k-j)+\ga][(1+\ga)(j+1)+\ga]-[(1+\ga)(k-j)-1][(1+\ga)j-1]\\
=&(1+\ga)[(1+\ga)k-1+\ga],
\end{align*}
and the statement follows for $k$ odd.
\par
Suppose $k$ is even.
We compute:
\begin{align*}
2q_j=&2(s_j-s_{j+1})
=\frac{(1+\ga)(k-j)+1}{(1+\ga)j-1}-\frac{(1+\ga)(k-j-1)+1}{(1+\ga)(j+1)-1}\\
=&\frac{[(1+\ga)(k-j)+1][(1+\ga)(j+1)-1]-[(1+\ga)(k-j-1)+1][(1+\ga)j-1]}{[(1+\ga)j-1][(1+\ga)(j+1)-1]}.
\end{align*}
Observing that $(1+\ga)(k-j-1)+1=(1+\ga)(k-j)-\ga$ and $(1+\ga)(j+1)-1=(1+\ga)j+\ga$,
we find
\begin{align*}
[(1+\ga)(k-j)+1][(1+\ga)(j+1)-1]-[(1+\ga)(k-j-1)+1][(1+\ga)j-1]=(1+\ga)^2k
\end{align*}
and the asserted decay rate follows.
\end{proof}
For the sake of completeness, we check the following fact which was stated in Section~\ref{sec:intro}.
\begin{rmk}
\label{rmk:massid}
The blow-up mass values $m_+(0),m_-(0)$, as defined in
\eqref{def:blowupmassodd}-- \eqref{def:blowupmasseven}, 
satisfy the mass identity \eqref{eq:massid}.
\end{rmk}
\begin{proof}
Throughout this proof, we set $\tmsp=\msp(0)/4\pi$, $\tmsm=\msm(0)/4\pi$.
Equivalently, we check that $\tmsp,\tmsm$ satisfy
\begin{equation}
\label{eq:redmassid}
2\left(\tmsp+\frac{\tmsm}{\ga}\right)=\left(\tmsp-\tmsm\right)^2
\end{equation}
Suppose $k$ is odd. Then,
\begin{equation*}
\begin{aligned}
\tmsp+\frac{\tmsm}{\ga}
=&\frac{1}{2}\left[(1+\frac{1}{\ga})k+1-\frac{1}{\ga}\right](k+1+\frac{1}{\ga}(k-1))\\
=&\frac{1}{2}\left[(1+\frac{1}{\ga})k+1-\frac{1}{\ga}\right]^2.
\end{aligned}
\end{equation*}
On the other hand,
\begin{equation*}
\tmsp-\tmsm=\left[(1+\frac{1}{\ga})k+1-\frac{1}{\ga}\right]
[\frac{k+1}{2}-\frac{k-1}{2}]=(1+\frac{1}{\ga})k+1-\frac{1}{\ga}.
\end{equation*}
Hence, \eqref{eq:redmassid} is verified for $k$ odd.
\par
Suppose $k$ is even.
\begin{equation*}
\tmsp+\frac{\tmsm}{\ga}
=k[(1+\frac{1}{\ga})\frac{k}{2}-\frac{1}{\ga}]+\frac{k}{\ga}[(1+\frac{1}{\ga})\frac{k}{2}+1]
=\frac{k^2}{2}(1+\frac{1}{\ga})^2.
\end{equation*}
On the other hand,
\begin{equation*}
\tmsp-\tmsm
=k[(1+\frac{1}{\ga})\frac{k}{2}-\frac{1}{\ga}]-k[(1+\frac{1}{\ga})\frac{k}{2}+1]
=-k(1+\frac{1}{\ga}),
\end{equation*}
and \eqref{eq:redmassid}  is verified for $k$ even, as well.
\end{proof}
\subsection{The cases of physical interest}
Finally, we compute the values of $\la,\tp$ for which the bubble  tower
construction as in Theorem~\ref{thm:main}
yields solutions to problem~\eqref{eq:Onsager} and to problem~\eqref{eq:Neri}.
Let $\ur$, $\rho\in(0,\rho_0)$ be the family of concentrating solutions as
obtained in Theorem~\ref{thm:main}.
\par
\textit{Onsager's mean field equation~\eqref{eq:Onsager}}.
The solution~$\ur$ yields a solution to \eqref{eq:Onsager}
with
\[
\begin{aligned}
&\la\tp=m_+(0),
&&\la(1-\tp)\ga=m_-(0).
\end{aligned}
\]
Consequently, recalling \eqref{eq:massid},
\[
\la=m_+(0)+\frac{m_-(0)}{\ga}=\frac{\calMk^2}{8\pi},
\]
and \eqref{eq:lambda}--\eqref{eq:Ontau} readily follow.
\par
\textit{Neri's mean field equation~\eqref{eq:Neri}}.
The solution~$\ur$ yields a solution to \eqref{eq:Neri}
satisfying
\[
\begin{aligned}
&\frac{\la\tp\int_\Omega e^{\ur}}{\tp\int_\Omega e^{\ur}+(1-\tp)\int_\Omega e^{-\ga{\ur}}}=m_+(0)+o(1)\\
&\frac{\la\ga(1-\tp)\int_\Omega e^{-\ga{\ur}}}{\tp\int_\Omega e^{\ur}+(1-\tp)\int_\Omega e^{-\ga{\ur}}}=m_-(0)+o(1).
\end{aligned}
\]
We deduce that
\[
\begin{aligned}
&[\la-m_+(0)+o(1)]\,\tp\int_\Omega e^{\ur}-[m_+(0)+o(1)]\,(1-\tp)\int_\Omega e^{-\ga{\ur}}=0\\
&[m_-(0)+o(1)]\,\tp\int_\Omega e^{\ur}-[\la\ga-m_-(0)+o(1)]\,(1-\tp)\int_\Omega e^{-\ga{\ur}}=0.
\end{aligned}
\]
A non-zero solution $(\tp\int_\Omega e^{\ur},(1-\tp)\int_\Omega e^{-\ga{\ur}})$ to this linear system exists if and only if
\[
(\la-m_+(0)+o(1))(\la\ga-m_-(0)+o(1))-(m_+(0)+o(1))(m_-(0)+o(1))=0.
\]
Taking limits and dividing by $\ga$ we obtain the condition
\[
(\la-m_+(0))(\la-\frac{m_-(0)}{\ga})-m_+(0)\frac{m_-(0)}{\ga}=0.
\]
Since $\la=m_+(0)+\ga^{-1}m_-(0)$ is the only positive solution to the second order algebraic equation above, we 
deduce that \eqref{eq:lambda} is satisfied for equation~\eqref{eq:Neri} as well.
\section*{Acknowledgements}
This work is supported by PRIN {\em Aspetti variazionali e perturbativi nei problemi differenziali nonlineari}.
A.P.\ is also supported by Fondi Sapienza 2015 \textit{Alcuni aspetti geometrici di equazioni ellittiche semilineari}.
T.R\ is also supported by 
Progetto GNAMPA-INDAM 2015: \emph{Alcuni aspetti di equazioni ellittiche non lineari}.

\end{document}